\documentclass[12pt]{article} 

\usepackage[width=180mm,top=25mm,bottom=25mm]{geometry} 

\usepackage[utf8]{inputenc}  

\usepackage{amssymb} 
\usepackage{mathptmx} 
\usepackage{mathtools} 
\usepackage{hyperref}
\usepackage[ocgcolorlinks]{ocgx2}
\usepackage{amsthm} 
\usepackage[noabbrev]{cleveref}
\usepackage{stmaryrd} 
\usepackage{amsmath}
\usepackage{todonotes}
\usepackage{comment}
\usepackage{color}
\usepackage{enumitem}
\usepackage{pgfplots} 
\usepackage{soul}

\newtheorem{prob}{Problem}[section] 
\newtheorem{lem}[prob]{Lemma} 
\newtheorem{thm}[prob]{Theorem} 
\newtheorem{defin}[prob]{Definition}
\newtheorem{memo}[prob]{Memo}
\newtheorem{prop}[prob]{Proposition}
\newtheorem{rem}[prob]{Remark}

\newtheorem{coro}[prob]{Corollary}
\newtheorem{examp}[prob]{Example}

\crefname{prob}{problem}{problems}
\crefname{lem}{lemma}{lemmas} 
\crefname{thm}{theorem}{theorems} 
\crefname{defin}{definition}{definitions}
\crefname{memo}{memo}{memos}
\crefname{prop}{proposition}{propositions}
\crefname{rem}{remark}{remarks}
\crefname{coro}{corollary}{corollaries}
\crefname{examp}{example}{examples}

\newcommand{\R}{\mathbb{R}}
\newcommand{\ws}{\stackrel{*}{\rightharpoonup}}
\newcommand{\mres}{\mathbin{\vrule height 1.3ex depth 0pt width 0.13ex\vrule height 0.13ex depth 0pt width 1ex}} 
\DeclareMathOperator*{\esssup}{ess\,sup}
\newcommand{\crefthmpart}[2]{\namecref{#1}~\hyperref[#2]{\labelcref*{#1}~\labelcref*{#2}}}
\newcommand{\cyan}[1]{\textcolor{black}{#1}}
\newcommand{\purple}[1]{\textcolor{black}{#1}}
\newcommand{\red}[1]{\textcolor{black}{#1}}
\newcommand{\J}{\cyan{\mathrm{D}}}
\newcommand{\e}{\cyan{\mathrm{d}}}

\makeatletter
\newsavebox{\@brx}
\newcommand{\llangle}[1][]{\savebox{\@brx}{\(\m@th{#1\langle}\)}%
  \mathopen{\copy\@brx\kern-0.5\wd\@brx\usebox{\@brx}}}
\newcommand{\rrangle}[1][]{\savebox{\@brx}{\(\m@th{#1\rangle}\)}%
  \mathclose{\copy\@brx\kern-0.5\wd\@brx\usebox{\@brx}}}
\makeatother

\usepackage[backend=biber,maxnames=5,style=numeric]{biblatex}
\title{Formulas for the $h$-mass on $1$-currents with coefficients in $\R^m$}
\author{Julius Lohmann\thanks{Department of Mathematics and Statistics, University of Helsinki, julius.lohmann@helsinki.fi}\qquad Bernhard Schmitzer\thanks{Institute for Computer Science, University of Göttingen, schmitzer@cs.uni-goettingen.de}\qquad Benedikt Wirth\thanks{Institute for Numerical and Applied Mathematics, University of Münster, benedikt.wirth@uni-muenster.de}}

\begin{document}
\setstcolor{cyan}
\maketitle
\noindent
\begin{abstract}
\cyan{We consider the minimization of the $h$-mass over normal $1$-currents in $\R^n$ with coefficients in $\R^m$ and prescribed boundary.}
\cyan{This optimization is known as multi-material transport problem and} used in the context of logistics of multiple commodities, but also as a relaxation of nonconvex optimal transport \cyan{tasks} such as so-called branched transport problems.
The $h$-mass \cyan{with norm $h$} can be defined in different ways, resulting in three functionals $\mathcal{M}_h,|\cdot|_H$, and $\mathbb{M}_h$, whose equality is the main result of this article:
$\mathcal{M}_h$ is a functional on $1$-currents in the spirit of Federer and Fleming,
norm $|\cdot|_H$ denotes the total variation of a Radon measure with respect to $H$ induced by $h$,
and $\mathbb{M}_h$ is a mass on flat $1$-chains in the sense of Whitney.
\red{On top we introduce }a new and improved notion of calibrations for the multi-material transport problem:
\cyan{w}e identify calibrations with \cyan{(weak) Jacobians of} optimizers of the associated convex dual problem,
which \cyan{yields their existence and natural regularity.} 
\end{abstract}
\textbf{Keywords}: calibrations, convex optimization, currents, geometric measure theory, $h$-mass, multi-material transport, optimal transport, Plateau problem, Wasserstein distance
\tableofcontents

\section{Introduction}
\label{intro}
The \textit{multi-material transport problem} was introduced by \textcite{MMT17}. In their work an Eulerian formulation (due to \textcite{G67} and \textcite{X}) was considered: Given $m$ (the number of different material types) sources $\mu_+=(\mu_+^1,\ldots,\mu_+^m)$ (discrete non-negative Radon measures) and sinks $\mu_-=(\mu_-^1,\ldots,\mu_-^m)$ on the spatial domain $\R^n$, find optimal mass fluxes $F_i$ (normal $1$-currents with boundaries $\partial F_i=\mu_-^i-\mu_+^i$) with respect to a function $h:\mathbb{Z}^m\to[0,\infty)$ describing the cost $h(\theta)$ to transport a material bundle $\theta$ per unit distance. Since the multiplicity $\theta$ lies in $\mathbb{Z}^m$, the only admissible structures for the transport are countably $1$-rectifiable currents \cite[Cor.\:on p.\:265]{W992}.
The multi-material transport model was generalized by \textcite{MMSR16} to arbitrary (possibly non-discrete) marginals $\mu_+$ and $\mu_-$ with cost function $h:\R^m\to[0,\infty)$.
Our main results will concern different equivalent reformulations of this multi-material transport problem.

In \cite{MMT17} a \textit{calibration} for an admissible candidate $m$-tuple $F=(F_1,\ldots,F_m)$ to the multi-material transport problem was defined as a certain $\R^m$-valued closed smooth differential $1$-form $\omega$ satisfying three conditions. Calibrations are a classical tool in geometric measure theory in the study of Plateau-type problems. If such a calibration exists, then $F$ is optimal \cite[Thm.\:1.13]{MMT17}. The main step in the classical argument is a variant of Stokes' theorem (one may view $F$ as a directed $\R^m$-weighted surface with prescribed boundary) which uses the closedness of $\omega$. In this article, we will, inter alia, see that the natural definition of a calibration is actually provided by the \textit{primal-dual extremality conditions} associated with the multi-material transport problem and its dual. This relation directly yields \cyan{the} existence of a calibration and its natural regularity as well as necessary \textit{and} sufficient optimality conditions for multi-material transport \cyan{(}as opposed to just sufficient conditions in case of the classical definition of a calibration). 

The multi-material transport problem is an optimization over normal $1$-currents with coefficients in $\R^m$ and prescribed boundary.
These objects can alternatively also be interpreted as $m$-tuples of divergence measure vectorfields \cite{Sil} ($\R^{m\times n}$-valued Radon measures whose row-wise distributional divergence is an $\R^m$-valued \cyan{Radon} measure) or as flat $1$-chains with coefficients in $\R^m$ having finite mass and boundary mass.
Each of these three interpretations of the decision variables, as currents, measures, and flat chains, comes with its own natural notion of an $h$-mass (which is the objective in the multi-material transport problem), resulting in three mass functionals $\mathcal{M}_h,|\cdot|_H$, and $\mathbb{M}_h$, where $h$ and $H$ are norms on $\R^m$ and $\R^{m\times n}$ respectively ($H$ is suitably defined via $h$). We will show that these functionals coincide. 

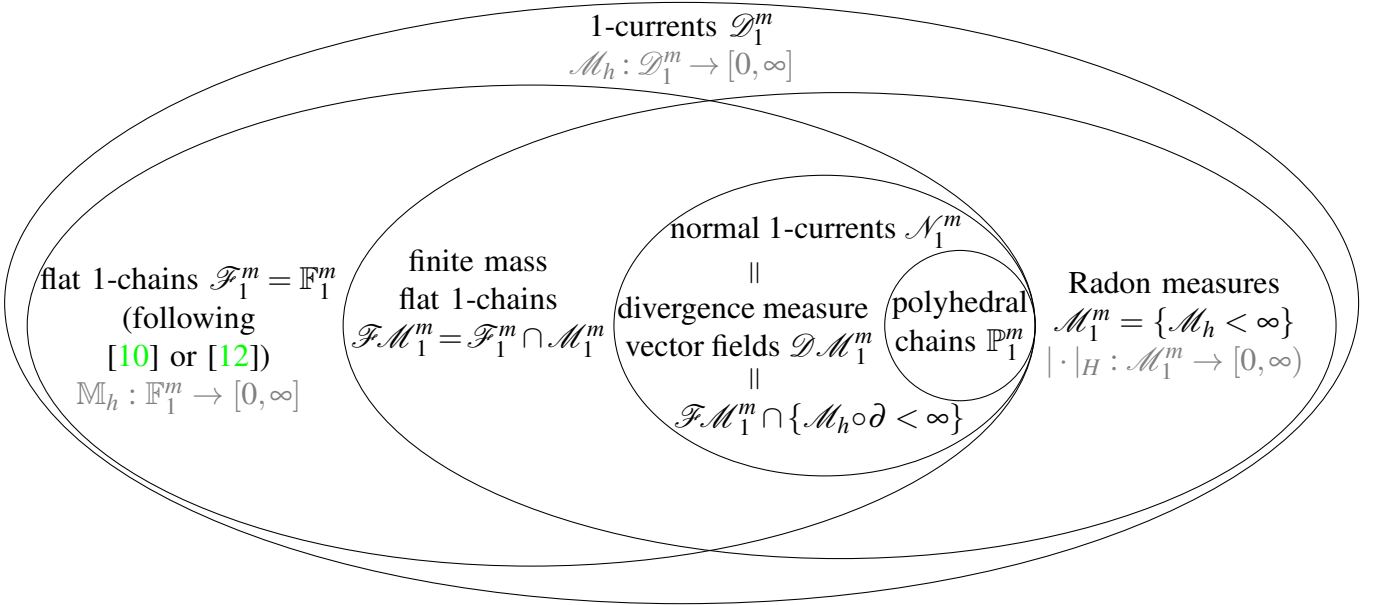
\begin{figure}
\centering
\begin{tikzpicture}
\draw (0,0) ellipse (9cm and 4cm);
\node[anchor=north,text width=3cm,align=center] at (0,4) {$1$-currents $\mathcal{D}_1^m$\\\color{gray}$\mathcal{M}_h:\mathcal{D}_1^m\to[0,\infty]$};

\draw (-2.,-.3) ellipse (6.7cm and 3.2cm);
\node[anchor=west,text width=4cm,align=center] at (-8.7,-.5) {flat $1$-chains $\mathcal{F}_1^m=\mathbb{F}_1^m$\\(following \cite{F} or \cite{F66})\\\color{gray}$\mathbb{M}_h:\mathbb{F}_1^m\to[0,\infty]$};

\draw (2.1,-.3) ellipse (6.6cm and 3.1cm);
\node[anchor=east,text width=4cm,align=center] at (8.7,-.3) {Radon measures $\mathcal{M}_1^m=\{ \mathcal{M}_h<\infty \}$\\\color{gray}$|\cdot|_H:\mathcal{M}_1^m\to[0,\infty)$};

\draw (1.9,-.3) ellipse (2.8cm and 2cm);
\node[anchor=west] at (-.3,.95) {normal $1$-currents $\mathcal{N}_1^m$};
\node at (1.,.4) {\rotatebox{90}{$=$}};
\node[anchor=west] at (-1.,-.1) {divergence measure};
\node[anchor=west] at (-.9,-.6) {vector fields $\mathcal{DM}_1^m$};
\node at (1.,-1.) {\rotatebox{90}{$=$}};
\node[anchor=west] at (-.25,-1.55) {$\mathcal{F\!\!M}_1^m\cap\{ \mathcal{M}_h\!\circ\!\partial<\infty \}$};

\draw (3.7,-.3) circle (1cm);
\node[text width=2cm,align=center] at (3.7,-.3) {polyhedral chains $\mathbb{P}_1^m$};

\node[anchor=west,text width=3.5cm,align=center] at (-4.6,.0) {finite mass flat $1$-chains $\mathcal{F\!\!M}_1^m\!=\!\mathcal{F}_1^m\cap\mathcal{M}_1^m$};
\end{tikzpicture}
\caption{Important subspaces of $\mathcal{D}_1^m =\mathcal{D}_1^m(\R^n)$, the space of $1$-currents in $\R^n$ with coefficients in $\R^m$.
The coefficient space $\R^m$ is indicated by the superscript $m$, while the spatial dimension $n$ is suppressed in the notation
(in particular, $\mathcal{M}_1^m$ represents $\R^{m\times n}$-valued Radon measures).
On their respective domain of definition, the three $h$-masses $\mathcal{M}_h,|\cdot|_H$, and $\mathbb{M}_h$ are displayed in grey.
The multi-material transport problem is posed on $\mathcal{N}_1^m$.}
\label{fig:overview}
\end{figure}

Let us provide a little more detail on the three functionals.
\Cref{fig:overview} can be referred to for an overview.
The definition of the \textit{mass} on $k$-currents goes back to \textcite{FF60} (cf.\:introduction in \cite{Wh99}). In \cite{FF60} the authors considered (compactly supported) $k$-currents (in the sense of DeRham) as (continuous and linear) functionals on $k$-forms and studied the coefficient groups $G=\mathbb{Z}$ and $G=\R$ (yielding the notion of rectifiable $k$-currents and real $k$-currents). The mass in \cite{FF60} is defined in the natural way according to this duality. This idea is used in \cite[Section 4.3]{MMSR16} to define a mass $\mathcal{M}$ for $m$-tuples of $k$-currents. For a general norm $h$ on $\R^m$, we will generalize this definition to the $h$-mass $\mathcal{M}_h$ on the space $\mathcal{D}_k^m$ of $k$-currents in $\R^n$ with coefficients in $\R^m$; for $k=1$ and $h_*$ the dual norm to $h$ we have
\begin{equation*}
\mathcal{M}_h(F)=\sup\left\{F(\omega)\:\middle|\:\omega\textup{ differential $1$-form with coefficients in $\R^m$ and }\sup_{x\in\R^n}\sup_{\cyan{\vec{e}}\in\R^n,|\cyan{\vec{e}}|\leq1}h_*(\omega(x)\cyan{\vec{e}})\leq 1\right\}.
\end{equation*}
We will then transform $h$ into a suitable matrix norm $H$ in such a way that $H(\theta\otimes\vec{e})=h(\theta)$ does not depend on unit direction $\vec{e}$.
This allows us to consider the total variation $|\cdot|_H$ on the space $\mathcal{M}_1^m$ of $\R^{m\times n}$-valued Radon measures (allowing for more general norms $H$ this functional may be used to study variants of multi-material transport in which the cost depends also on the direction of transport),
\begin{equation*}
|F|_H
=\sup\left\{\sum_{i}H(F(B_i))\:\middle|\:B_1,B_2,\ldots\textup{ is a countable measurable partition of }\R^n\right\}
=\int H\left( \frac{\e F}{\e |F|} \right)\,\mathrm{d}|F|.
\end{equation*}
Note that the latter equality follows from a classical argument \cite[Ch.\:2, Thm.\:4 (iv)]{DU77}.
In \cite{F66} Fleming used Whitney's concept of $k$-chains \cite[Ch.\:5]{W57} to define flat $k$-chains with coefficients in an arbitrary complete metric space $(G,d)$ as completion of the space of polyhedral $k$-chains with respect to Whitney's flat distance. The related mass functional is defined via relaxation of the mass on polyhedral $k$-chains (in which $d(\cdot,0)$ is applied to the coefficients). The multi-material case corresponds to the choice $G=\R^m$ and $d$ induced by $h$, so the corresponding mass functional on the space $\mathbb{F}_1^m$ of flat $1$-chains with coefficients in $\R^m$ is given by
\begin{equation*}
\mathbb{M}_h(F)
=\inf\left\{\liminf_{i}\:\mathbb{M}_h(P_i)\:\middle|\:\textup{sequence $P_i$ of polyhedral $1$-chains converges to $F$ in \cyan{the} flat sense}\right\},
\end{equation*}
where $\mathbb{M}_h(P)=\sum_{j=1}^{N}h(\theta_j)\mathcal{H}^1(s_j)$
for a polyhedral $1$-chain $P=\sum_{j=1}^{N}\theta_j\otimes\vec{s_j}\mathcal{H}^1\mres s_j$ with straight line segments $s_j\subset\R^n$ and corresponding unit tangent vectors $\vec{s_j}$
($\mathcal{H}^1\mres s$ denotes the one-dimensional Hausdorff measure restricted to $s$).

We will show $\mathcal{M}_h=|\cdot|_H$ on $\mathcal{M}_1^m$ and $|\cdot|_H=\mathbb{M}_h$ on the space $\mathcal{F\!\!M}_1^m$ of flat $1$-chains with coefficients in $\R^m$ and finite mass
(which implies $\mathcal{M}_h=\mathbb{M}_h$ on all of \cyan{$\mathbb{F}_1^m$}).
For the case $m=1$ and $h=\textup{abs}$ being the absolute value these relations are well-known \cite[Thm.\:4.1.23 \&\:Lem.\:4.2.23]{F},
and the corresponding single-material transport problem $\min_{\partial F=\mu_--\mu_+}\:\mathbb{M}_\textup{abs}(F)$ is exactly the so-called Beckmann formulation of the Wasserstein-$1$ distance between $\mu_+$ and $\mu_-$. Applying Fenchel\textendash Rockafellar duality to this problem yields the so-called Kantorovich\textendash Rubinstein formula,
\begin{equation*}
\min_{\partial F=\mu_--\mu_+}\:\mathbb{M}_\textup{abs}(F)=\max_\phi\int\phi\,\mathrm{d}(\mu_+-\mu_-),
\end{equation*}
where the maximum is over Lipschitz continuous Kantorovich potentials $\phi$ with $|\nabla\phi|\leq 1$ almost everywhere. We will generalize this formula for the case $m>1$ and general $h$. The corresponding primal-dual extremality conditions will then yield our definition of calibrations.


To conclude, we remark that the duality viewpoint is inherent in the definitions of classical mass $\mathcal{M}_\textup{abs}$ and dual flat seminorms in \cite[pp.\:358 \&\:367]{F} (duality of $k$-currents and compactly supported smooth differential $k$-forms) or in the dual formulations of the Wasserstein-$1$ distance (Kantorovich\textendash Rubinstein formula and Beckmann formulation), and it was roughly related to calibrations in the text passage on \cite[p.\:3]{MM14}. We will see that our new definition \cyan{of calibrations} via primal-dual extremality conditions is the most natural one: \cyan{i}t provides sufficient \textit{and} necessary conditions, it includes the classical definition,
and it reveals the generic regularity of calibrations.

In the remainder of this section, we introduce the terminology used in multi-material transport and summarize our main results
(the involved notions will be rigorously introduced only later, but the results should be understandable nevertheless).
In \cref{preliminaries} we list the used notation and definitions (in particular, the different notions of flat $k$-chains with coefficients in $\R^m$) as well as some auxiliary statements. \Cref{dualitysec} will be devoted to duality in multi-material transport including the generalized Kantorovich\textendash Rubinstein formula. An import subresult will be that $\mathbb{M}_h=|\cdot|_H$ on the set of minimizers of the multi-material transport problem. It will be used in \cref{equalityofmasses} where the equality of the $h$-masses will be shown.

\subsection{Multi-material transport}
\label{MMT}
The multi-material transport problem was introduced (with coefficients in $\mathbb{Z}^m$) in \cite{MMT17} and generalized (to coefficients in $\mathbb{R}^m$) in \cite{MMSR16}.
It is the \cyan{transport} task \cyan{of finding} the optimal simultaneous flux of $m$ materials from a given source to a given sink distribution.
These distributions must of course admit a connecting flux.

\begin{defin}[Source and sink distribution]
\label{boundaryprop}
An admissible \textbf{source and sink distribution} is a compactly supported $\R^m$-valued Radon measure $F_0$ on $\R^n$
such that $F_0(\R^n)=0$ or equivalently there exists some normal $1$-current $F_1$ in $\R^n$ with coefficients in $\R^m$ and $\partial F_1=F_0$.
\end{defin}
\cyan{A source and sink distribution $F_0$ can be written $F_0=\mu_--\mu_+$ (Hahn decomposition theorem), where $\mu_+$ and $\mu_-$ can be interpreted as the source and sink configurations of the material bundles $\theta\cyan{\in\R^m}$ (with non-negative components and satisfying $\mu_+(\R^n)=\mu_-(\R^n)$).}

Optimality is with respect to an overall transporation cost,
which in turn depends on the cost $h(\theta)$ to transport a material bundle $\theta\in\R^m$ per unit distance.
In this article, $h$ will have the following regularity.
\begin{defin}[Multi-material transport cost]
	\label{mtc}
	A \textbf{multi-material transport cost} is a norm $h:\R^m\to[0,\infty)$.
\end{defin}
Norms and their properties appear quite naturally as cost in this context.
The subadditivity leads to the appearance of branched structures in the support of optimal multi-material mass fluxes: 
\cyan{t}ypically, the cost $h(\theta+\bar\theta)$ for transporting material bundles $\theta$ and $\bar\theta$ together
is strictly smaller than the cost $h(\theta)+h(\bar\theta)$ for transporting them separately
so that the overall cost can be reduced by merging multiple initially separate mass fluxes.
The absolute homogeneity induces a scaling property: if $F_{\textup{opt}}$ is a minimizer of the $h$-mass subject to $\partial F=F_0$, then $\lambda F_{\textup{opt}}$ is a minimizer under $\partial F=\lambda F_0$ for all $\lambda\in\R$. Finally, the definiteness is a natural property of a transport cost.

A multi-material transport cost $h$ induces a second norm $H$ that becomes relevant when viewing multi-material mass fluxes $F$ as $\R^{m\times n}$-valued Radon measures (each row describing the mass flux of one of the material types). This norm is applicable to densities of $F$ (e.g.\:the Radon\textendash Nikodym derivative with respect to the total variation measure $|F|$).
\begin{defin}[Generated norm]
	\label{generated}
	The norm $H:\R^{m\times n}\to[0,\infty)$ \textbf{generated by} a multi-material transport cost $h:\R^m\to[0,\infty)$ is defined by specifying its subdifferential in zero as
	\begin{equation*}
	\partial H(0)=\left\{  M\in\R^{m\times n}\:\middle|\: MB_1(0)\subset\partial h(0)\right\}.
	\end{equation*}
\end{defin}
Above, $B_1(0)\subset\R^n$ denotes the open unit ball (due to closedness of $\partial h(0)$ one may just as well use the closed ball).
This definition is motivated by the following simple fact, proven in \cref{basics}.

\begin{prop}[Generated norm as envelope]\label{propertygenerated}
The norm $H$ generated by $h$ is the largest convex function on $\R^{m\times n}$ satisfying $H(\theta\otimes\vec e)=h(\theta)$ for all $\theta\in\R^m$ and unit vectors $\vec e\in\mathcal{S}^{n-1}\subset\R^n$.
\end{prop}

This essentially means that $H$ is independent of the transport direction.
For example, if $F$ is rectifiable, that is, $F=\theta\otimes\vec{e}\mathcal{H}^1\mres S$, where multiplicity $\theta\in\R^m$ is integrable with respect to the one-dimensional Hausdorff measure $\mathcal{H}^1$ restricted to a countably $1$-rectifiable set $S$ and $\vec{e}\in\R^n$ is a tangential unit vectorfield defined $\mathcal{H}^1$-a.e.\:on $S$, then \cyan{by \cref{propertygenerated}} the \cyan{total variation} of $F$ \cyan{with respect to $H$ is equal to} $|F|_H=\int_SH(\theta\otimes\vec{e})\,\mathrm{d}\mathcal{H}^1=\int_Sh(\theta)\,\mathrm{d}\mathcal{H}^1$. Note that in principle one could also consider more general norms $H$ which are not generated by any $h$ and therefore do not necessarily need to be independent of the transport direction.

With these norms we can give \cyan{three} formulations of the multi-material transport problem (which will turn out to be equivalent by \cyan{\cref{equalitymasses}}).
Recall that, if $F$ is identified with a $\R^{m\times n}$-valued Radon measure, then $\partial F$ equals the negative distributional divergence (applied row-wise).

\begin{defin}[Three models for multi-material transport]
	Let $F_0$ be an admissible source and sink distribution. Fix a multi-material transport cost $h:\R^m\to[0,\infty)$ and assume that $H$ is generated by $h$. The \textbf{multi-material transport problem} is any of the optimization problems
	\begin{gather}
	\label{DM}\tag{$\mathcal{DM}$}
	\inf_F\:|F|_H,\\
	\label{M}\tag{$\mathcal{M}$}
	\inf_F\:\mathcal{M}_h(F),\\
	\label{F}\tag{$\mathbb{F}$}
	\inf_F\:\mathbb{M}_h(F),
	\end{gather} 
	where infimizations are over normal $1$-currents $F$ in $\R^n$ with coefficients in $\R^m$ subject to $\partial F=F_0$.
\end{defin}
We will see that all three are equivalent minimization problems (\cyan{\cref{equalitymasses}}).
As explained before, $|\cdot|_H$ is the total variation with respect to the norm $H$ on $\R^{m\times n}$ when $F$ is viewed as a $\R^{m\times n}$-valued Radon measure.
The label \labelcref{DM} indicates that $|\cdot|_H$ is applied to objects whose components (or rows) are divergence measure vectorfields \cite{Sil}.
In contrast, $\mathcal{M}_h$ is an $h$-mass defined in the spirit of \textcite{FF60}.
Finally, $\mathbb{M}_h$ denotes the $h$-mass defined via relaxation of the $h$-mass on polyhedral $1$-chains in $\R^n$ with coefficients in $\R^m$ with respect to flat convergence.
It originates from a work of \textcite{W57} which was used by \textcite{F66} to define a mass on flat $k$-chains with weights in an arbitrary coefficient group.

\subsection{Application of multi-material transport for branched transport}
The multi-material transport model was used to convexify a class of \textit{branched transport} problems. In the branched transport problem, one aims to find an optimal mass flux $\mathcal{F}$ (\red{scalar} normal $1$-current) between two (compactly supported) non-negative Radon measures $\nu_+$ and $\nu_-$ with equal mass ($\mathcal{F}$ has to satisfy $\partial\mathcal{F}=\nu_--\nu_+$). The optimality is assessed with respect to a \textit{concave} transportation cost $\tau$ describing the effort $\tau(\theta)$ to move an amount of mass $\theta\in[0,\infty)$ per unit distance (with $\tau(0)=0$). For $\tau(\theta)=\theta^\alpha$ ($\alpha\in[0,1]$ fixed) this Eulerian formulation was introduced by \textcite{X}. At the same time, a Lagrangian formulation was established by \textcite{MSM03}. Both were generalized to arbitrary $\tau$ in \cite{BW}. To the best of found knowledge, the first attempts to reformulate branched transport problems as variants of multi-material transport were in \cite{MM14,MM16} based on \labelcref{M}. In \cite{MM14} the Steiner problem, which has already been studied in the 19th century by Gergonne (1811) and Gauss (1836) (see \cite{SS10} and \cite{BGTZ14}), of finding a network with minimal total length connecting a finite number of points in $\R^n$ (it can be modelled using $\tau\equiv 1$ on $ (0,\infty)$) was rephrased as a problem on rectifiable $1$-currents with coefficients in some discrete group $G$ \cite[Thm.\:2.4]{MM14}. In \cite{MM16} the Gilbert\textendash Steiner problem, whose name derives from a generalization of the Steiner problem by Gilbert (1967) \cite{G67}, was considered: find an optimal rectifiable $1$-current (coefficients in $\mathbb{Z}$) between $m$ source and $m$ sink points with respect to $\tau(\theta)=\theta^\alpha$ ($\alpha\in (0,1)$). It was rewritten as a multi-material transport problem with coefficient group $\mathbb{Z}^m$. This reformulation and the definition of calibrations allowed the authors of \cite{MM16} to prove the optimality of an irrigation tree (one sink, coefficients in $\mathbb{Q}$) in the infinite-dimensional spatial domain $\ell^2$ having infinitely many branchings.

\subsection{Main results}
The main results of this article are
\begin{itemize}
	\item a generalized Kantorovich\textendash Rubinstein formula $\min_{\partial F=F_0}\:\mathbb{E}(F)=\max_\phi\int\phi\cdot\mathrm{d}F_0$, where $\mathbb{E}\in\mathcal{E}=\{ \mathcal{M}_h,|\cdot|_H,\mathbb{M}_h \}$,
	\item a natural definition of calibrations for multi-material transport via extremality conditions,
	\item equality of the functionals in $\mathcal{E}$ (more precisely, given two functionals in $\mathcal{E}$ they coincide on the intersection of their domains of definition), and resulting from that
	\item an integral representation of $\mathcal{M}_h$ and $\mathbb{M}_h$.
\end{itemize}
Below we provide a little more detail on these results; \cref{fig:overview} helps as a reference (see \cref{currentschains} for the rigorous definitions of all spaces and functionals).
We write $\mathcal{N}_k^m(\R^n)$ for the space of normal $k$-currents in $\R^n$ with coefficients in $\R^m$, then $\mathcal{N}_0^m(\R^n)$ coincides with the \purple{(compactly supported)} $\R^m$-valued Radon measures.
Let $F_0\in \mathcal{N}_0^m(\R^n)$ be an admissible source and sink distribution according to \cref{boundaryprop}. Fix a multi-material transport cost $h:\R^m\to[0,\infty)$ and assume that $H$ is generated by $h$. 
We first prove equivalence of the multi-material transport problems \labelcref{DM} and \labelcref{F} by passing to the dual optimization problem.
\begin{thm}[Equivalence of \labelcref{DM,F}, generalized Kantorovich\textendash Rubinstein formula]
\label{EquivalenceKR}
\labelcref{DM,F} are equivalent minimization problems and can be written as maximization over generalized Kantorovich potentials,
\begin{equation*}
\min_F\:|F|_H=\min_F\:\mathbb{M}_h(F)=\max_\phi \int \phi\cdot\mathrm{d }F_0,
\end{equation*}
where $F\in\mathcal{N}_1^m(\R^n)$ with $\partial F=F_0$ and $\phi\in C^{0,1}(\R^n;\R^m)$ with $\J\phi\in\partial H(0)$ a.e.
Minima and maxima are attained.
\end{thm}

Motivated by \cref{EquivalenceKR} we give the following definition of a calibration. 
\begin{defin}[Calibration]
\label{caldef}
We say that $\Phi$ is a \textbf{calibration} for $F$ if $\Phi=\J\phi$ and the pair $(F,\phi)$ satisfies the \textit{primal-dual extremality conditions} associated with problem \labelcref{F} and its dual.
These can e.g.\:be stated as
\begin{itemize}
\item feasibility of $(F,\phi)\in\mathcal{N}_1^m(\R^n)\times C^{0,1}(\R^n;\R^m)$, i.e.\:$\partial F=F_0$ and $\J\phi\in\partial H(0)$ a.e.,
\item equality of the primal and dual functional values, i.e.\:$\cyan{\mathbb{M}_h(F)}=\int \phi\cdot\mathrm{d }F_0$.
\end{itemize}
\end{defin}
\begin{rem}[Extremality conditions]
Feasibility of the primal-dual pair and equality of the primal and dual functional values are the simplest way to express the primal-dual optimality conditions.
If strong duality holds and both primal and dual problem admit an optimizer (as in our case by \cref{EquivalenceKR}),
they are necessary and sufficient for the optimality of $F$ as well as of $\phi$.
An equivalent alternative would be to use the standard first order conditions of convex analysis given in \cite{R67} (see \cref{cwJ}).
\end{rem}
The following statement is a trivial implication.
\begin{coro}[Existence of calibration]
\label{excal}
If $\Phi=\J\phi$ is a calibration for $F$, then $(F,\phi)$ is an optimal pair for problem \labelcref{F} and its dual. Conversely, the weak Jacobian of any optimal $\phi$ is a calibration for all optimal $F$, and such a calibration always exists.
\end{coro}
Let us briefly compare our definition of a calibration with the one given in \cite{MMT17}. A calibration according to \cite[Def.\:1.12]{MMT17} for a rectifiable $F=\theta\otimes\vec{e}\mathcal{H}^1\mres S$ with multiplicity $\theta\in\mathbb{Z}^m$ (the representation of $F$ is as in \cref{MMT}) would be a smooth differential $1$-form $\omega\in\Omega_m^1(\R^n)=C^\infty(\R^n;\R^{m\times n})$ such that
\begin{enumerate}[label=C.\arabic*]
\item $(\omega(x)\vec{e}(x))\cdot\theta(x)=h(\theta (x))$ for $\mathcal{H}^1$-a.e.\:$x\in S$,\label{c1}
\item $\e\omega=\textup{curl}(\omega)=0$\todo[disable]{should we always use $\mathrm d$?} (weak curl applied row-wise, cf.\:\cite[p.\:8]{Sil}),\label{c2}
\item $(\omega(x)\vec{u})\cdot\vartheta\leq h(\vartheta)$ for all $x\in\R^n,\vartheta\in\R^m$, and unit vectors $\vec{u}\in\cyan{\mathcal{S}^{n-1}}$.\label{c3}
\end{enumerate}
\cyan{Such an} $\omega$ \cyan{(if it exists)} also \cyan{calibrates} in the sense of \cref{caldef} (and therefore certifies optimality of $F$):
Indeed, the closedness condition \ref{c2} is equivalent to the existence of a differential $0$-form $\phi\in\Omega_m^0(\R^n)=C^\infty(\R^n;\R^m)$ with $\omega=\cyan{\e\phi=\J\phi}$.
Condition \ref{c3} then is equivalent to the feasibility condition $\J\phi\in\partial H(0)$ a.e.,
and condition \ref{c1} then is equivalent to the equality of primal and dual functional values due to \cyan{($:$ denotes the Frobenius inner product on matrices)}
\begin{equation*}
\purple{\mathbb{M}_h(F)=|F|_H}
=\int_Sh(\theta)\,\mathrm d\mathcal{H}^1
=\int_S(\omega(x)\vec{e}(x))\cdot\theta(x)\,\mathrm d\mathcal{H}^1
=\int\omega:\e F
=\int\phi\cdot\e (\partial F)
=\int\phi\cdot\mathrm dF_0,
\end{equation*}
where \cref{equalitymasses} was used in the first equation.
Obviously, the difference of the above definition to \cref{caldef} is merely the regularity:
\cyan{w}hile above, $\omega$ or equivalently $\phi$ is required to be infinitely differentiable, $\phi$ is only required Lipschitz in \cref{caldef}.
(As a consequence of this reduced regularity, condition \ref{c1} can no longer be applied since $\J\phi$ might not be defined on $S$,
but it \cyan{can} be reformulated using a smoothing procedure, see \cite[Thm.\:3.2]{Sil}\cyan{.})
Therefore, \cref{caldef} allows many more candidates as calibrations,
and the existence statement in \cref{EquivalenceKR} implies that the regularity required in \cref{caldef} is just the right one:
a calibration for optimal normal currents $F$ (not necessarily integral) always exists (which is not the case for the $\omega$ from above).
Therefore, as already mentioned before, our calibration represents a necessary \textit{and} sufficient (as opposed to just sufficient) condition for optimality in \labelcref{F}.

\red{We next turn to the equality of the different mass functionals.}
\begin{thm}[Equality of masses $\mathcal{M}_h,|\cdot|_H,\mathbb{M}_h$]
\label{equalitymasses}
We have 
\begin{center}
\begin{minipage}{0.5\textwidth}
\begin{enumerate}[label=\textup{M.\arabic*}]
\item \qquad $\mathcal{M}_h=|\cdot|_H$\qquad on\qquad $\mathcal{M}_1^m(\R^n)$\qquad and\label{item1}
\item \qquad $|\cdot|_H=\mathbb{M}_h$\qquad on\qquad $\mathcal{F\!\!M}_1^m(\R^n)$.\label{item2}
\end{enumerate}
\end{minipage}
\end{center}
\end{thm}

Let us briefly provide the idea for proving $\mathbb{M}_h\leq |\cdot|_H$.
\red{\Cref{EquivalenceKR} directly implies $\mathbb{M}_h(F)\leq|F|_H$ for all minimizers $F\in\mathcal{N}_1^m(\R^n)$ of \labelcref{F},
which together with the (straightforward) inequality $\mathbb{M}_h(F)\geq|F|_H$ implies that }\purple{all minimizers $F$ of \labelcref{F} satisfy $\mathbb{M}_h(F)=|F|_H$ \red{(and thus also minimize \labelcref{DM}).}}
Now assume that $F\in\mathcal{N}_1^m(\R^n)$ is arbitrary. We will use suitable $\delta$-partitions $(P_\delta^i)\subset\R^n$ of $\R^n$ to construct $F_\delta\in\mathcal{N}_1^m(\R^n)$ such that its restriction $F_i=F_\delta\mres P_\delta^i$ to $P_\delta^i$ is optimal for \labelcref{F} subject to the boundary constraint $\partial F_i=\partial(F\mres P_\delta^i)$. We will then show that $F_\delta\to F$ for $\delta\to 0$ with respect to flat convergence. The optimality of $F_\delta\mres P_\delta^i$ and the lower semicontinuity of $\mathbb{M}_h$ will then yield the desired inequality $\mathbb{M}_h(F)\leq |F|_H$. The case $F\in\mathcal{F\!\!M}_1^m(\R^n)$ can be reduced to $F\in\mathcal{N}_1^m(\R^n)$ by applying the deformation theorem.

Note that \cref{equalitymasses} yields an integral representation for the functionals $\mathbb{M}_h$ and $\mathcal{M}_h$. In particular, if $F\in\mathcal{F}_1^m(\R^n)$ has finite mass, then by \cite[Thm.\:5.5]{Sil} we have the decomposition $F=F_{\textup{rect}}+F_{\textup{diff}}$ into a rectifiable part $F_{\textup{rect}}=\theta\otimes\vec{e}\mathcal{H}^1\mres S$ and a diffuse part $F_{\textup{diff}}$, i.e.\:$F_{\textup{diff}}$ is singular with respect to $\mathcal{H}^1\mres R$ for all countably $1$-rectifiable $R\subset\R^n$. 
Hence we get
\begin{equation*}
\mathcal{M}_h(F)
=\mathbb{M}_h(F)
=\int_S h(\theta)\,\mathrm{d}\mathcal{H}^1+\int H\left( \frac{ \e F_{ \textup{diff} }}{\e |F_{\textup{diff} }| } \right)\,\mathrm{d}|F_{\textup{diff}}|.
\end{equation*}
Note that $\frac{ \e F_{ \textup{diff} }}{\e |F_{\textup{diff} }| }$ may have rank larger than one so that the second integral cannot be expressed in terms of $h$.
By \cref{propertygenerated}, though, one may write
\begin{equation*}
H(M)=\inf\left\{h(\theta_1)+\ldots+h(\theta_i)\,|\,M=\theta_1\otimes\vec e_1+\ldots+\theta_i\otimes\vec e_i\text{ for }i\in\mathbb{N},\theta_1,\ldots,\theta_i\in\R^m,\vec e_1,\ldots,\vec e_i\in\cyan{\mathcal{S}^{n-1}}\right\}.
\end{equation*}
For example, if $m=n=2$ and $h(\theta)=\max\{|\theta_1|,|\theta_2|,|\theta_1-\theta_2|\}$, then, defining
\begin{equation*}
\vec e_1=\frac1{2\sqrt2}{1+\sqrt3\choose1-\sqrt3},\
\vec e_2=\frac1{2\sqrt2}{1-\sqrt3\choose1+\sqrt3},\
\vec e_3=\frac1{\sqrt2}{1\choose1},\
t_1=\sqrt{\frac23}{1\choose0},\
t_2=\sqrt{\frac23}{0\choose1},\
t_3=\frac{\sqrt3-1}{\sqrt6}{1\choose1},
\end{equation*}
the identity matrix $I$ satisfies
\begin{multline*}
H(I)
=H(t_1\otimes\vec e_1+t_2\otimes\vec e_2+t_3\otimes\vec e_3)
\leq H(t_1\otimes\vec e_1)+H(t_2\otimes\vec e_2)+H(t_3\otimes\vec e_3)
=h(t_1)+h(t_2)+h(t_3)\\
=\tfrac{1+\sqrt3}{\sqrt2}
=\underbrace{\tfrac1{2\sqrt2}\left(\begin{smallmatrix}1+\sqrt3&1-\sqrt3\\1-\sqrt3&1+\sqrt3 \end{smallmatrix}\right)}_{\overline M}\cyan{:}I
\leq\sup_{M\in\partial H(0)} M\cyan{:}I
=H(I)
\end{multline*}
so that all inequalities are actually equalities.
This indicates that the cost $H(I)$ arises as the homogenization of flows of material bundles proportional to $t_1,t_2,t_3$ in directions $\vec e_1,\vec e_2,\vec e_3$, respectively, see \cref{fig:homogenization}.

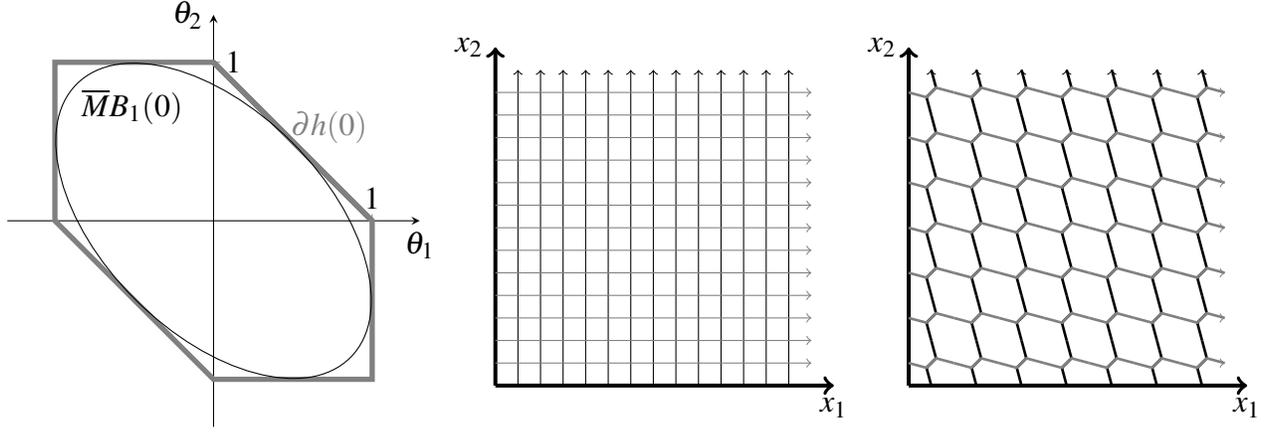
\begin{figure}
\centering
\begin{tikzpicture}
\begin{axis}[
x = 5em, y = 5em,
ymax = 1.3,
ymin = -1.3,
xmax = 1.3,
xmin = -1.3,
axis x line=center,
axis y line=center,
xlabel = {$\theta_1$}, 
ylabel = {$\theta_2$}, 
x label style={at={(axis description cs:1,.5)},anchor=north},
y label style={at={(axis description cs:.5,1)},anchor=east},
x tick label style={anchor=south, outer sep=3pt},
y tick label style={anchor= west, outer sep=3pt},
xtick = {1},
xticklabels = {$1$},
ytick = {1},
yticklabels = {$1$},
major tick length = 5pt,
]
\addplot[gray,line width=2pt] coordinates {(1,0) (0,1) (-1,1) (-1,0) (0,-1) (1,-1) (1,0)(0,1)};
\draw (axis cs:0,0) ellipse [rotate=45,x radius=3.5em,y radius=6.1em];
\node[gray,anchor=south west] at (axis cs:0.42,0.42) {$\partial h(0)$};
\node[anchor=north west] at (axis cs:-.9,.9) {$\overline MB_1(0)$};
\end{axis}
\end{tikzpicture}%
\begin{tikzpicture}
\draw[->,line width=1.5pt] (0,0) -- (4.5,0) node[anchor=north] {$x_1$};
\draw[->,line width=1.5pt] (0,0) -- (0,4.5) node[anchor=east] {$x_2$};
\foreach \x in {0.3,0.6,...,4.2} \draw[->] {(\x,0) -- (\x,4.2)};
\foreach \y in {0.3,0.6,...,4.2} \draw[->,gray] {(0,\y) -- (4.2,\y)};
\end{tikzpicture}%
\begin{tikzpicture}
\tikzset{
  microstructure/.pic={
    \draw[line width=1pt] (0,-.3) -- (-.0634,-.0634) -- (.0634,.0634) -- (0,.3);
    \draw[gray,line width=1pt] (-.3,0) -- (-.0634,-.0634) -- (.0634,.0634) -- (.3,0);
  }
}
\draw[->,line width=1.5pt] (0,0) -- (4.5,0) node[anchor=north] {$x_1$};
\draw[->,line width=1.5pt] (0,0) -- (0,4.5) node[anchor=east] {$x_2$};
\foreach \x in {0.3,0.9,...,4.2}
	\foreach \y in {0.3,0.9,...,4.2}
		\pic at (\x,\y) {microstructure};
\foreach \x in {0.3,0.9,...,4.2}
	\draw[->] (\x,4.199) -- (\x,4.2);
\foreach \y in {0.3,0.9,...,4.2}
	\draw[->,gray] (4.199,\y) -- (4.2,\y);
\end{tikzpicture}%
\caption{Microstructure underlying a diffuse flux in $n=2$ dimensions with $m=2$ materials.
Left: $\partial h(0)$ for $h(\theta)=\max\{|\theta_1|,|\theta_2|,|\theta_1-\theta_2|\}$ as well as $\overline MB_1(0)\subset\partial h(0)$.
Middle: Macroscopic view of the diffuse flux $F=I\mathcal L^2\mres[0,1]^2$ (for $I$ the identity matrix and $\mathcal L^2$ the Lebesgue measure)
with the first material flowing \cyan{to the right (grey)} and the second \cyan{upwards (black)}.
Right: Microscopic view of the flux, showing how the microscopic flow pattern of the black and grey material produce the macroscopic flux $F$ at cost $H(I)$.}
\label{fig:homogenization}
\end{figure}

\Cref{equalitymasses} already implies $\mathcal{M}_h=\mathbb{M}_h$ on $\mathcal{F\!\!M}_1^m(\R^n)$, which \cyan{directly yields equality on} $\mathcal{F}_1^m(\R^n)$.

\begin{coro}[Equality of $\mathcal{M}_h$ and $\mathbb{M}_h$]
\label{equalmassflat}
$\mathcal{M}_h=\mathbb{M}_h$ on $\mathcal{F}_1^m(\R^n)$.
\end{coro}

\begin{rem}[Non-compact support]
\label{nocompact}
For technical reasons, as in \cite[Sec.\:4.1.12]{F}, we define the elements of $\mathcal{F}_1^m(\R^n)$ to be compactly supported,
which is also done in other works (e.g.\:by White, see remark on \cite[p.\:162]{Wh99} or \cite{W992}, or by	 Smirnov, see definition of the subset of normal charges on \cite[p.\:842]{S}).
Nevertheless, our results stay true without this assumption. Indeed, letting $F$ a flat $1$-chain with coefficients in $\R^m$ and finite mass (possibly with unbounded support), application of \cite[Sec.\:4]{F66} yields the existence of a Radon measure $\mu_F:\mathcal{B}(\R^n)\to[0,\infty)$ with $\mu_F(B)=\mathbb{M}_h(F\mres B)$ for all $B\in\mathcal{B}(\R^n)$. Picking a countable Borel partition $(B_i)\subset\mathcal{B}(\R^n)$ with each $B_i$ being bounded, we obtain
\begin{equation*}
\mathbb{M}_h(F)=\mu_F(\R^n)=\sum_i\mu_F(B_i)=\sum_i\mathbb{M}_h(F\mres B_i)=\sum_i\mathcal{M}_h(F\mres B_i)=\mathcal{M}_h(F),
\end{equation*}
where \cref{equalmassflat} was used in the fourth equation and the $\sigma$-additivity of $\mathcal{M}_h$ (see \crefthmpart{equalitymasses}{item1}) in the last equation.
\end{rem}

\section{Preliminaries}
\label{preliminaries}
In this section\cyan{,} we introduce our notation and recapitulate the notions of currents and relevant subsets.
\subsection{Basic notation}\todo[inline,disable]{notations that are only locally used in one place I would explain there -- perhaps $f_\varepsilon$, dist, effective domain, adjoint, affine hull?}
\label{basicnot}
Let $k\in\mathbb{N}$. One may identify $\R^k=\R^{k_1\times k_2}$ if $k=k_1k_2$. We will use the following standard notation.
\begin{itemize}
		\item $m,n$: \textbf{material} and \textbf{spatial dimension}
		\item $\mathcal{S}^{n-1}$: \textbf{unit sphere} in $\R^n$
		\item $x\cdot y,M:N$: \textbf{Euclidean inner product} of vectors and \textbf{Frobenius inner product} of matrices
		\item $x\otimes y$: \textbf{dyadic product} 
		\item $|.|,h_*$: \textbf{Euclidean} or \textbf{Frobenius norm}, \textbf{dual norm} of a norm $h$
		\item $(x_i)\subset S$: \textbf{sequence} $x:\mathbb{N}\to S$ of elements in some set $S$ with $x_i=x(i)$
		\item $B_r(x)$: open \textbf{Euclidean ball} with radius $r>0$ and centre $x\in\R^n$
		\item $[x,y]$: \textbf{line segment} $x+[0,1](y-x)$, where $x,y\in\R^n$
		\item $\iota_S$: \textbf{indicator function} for set $S$,
		\begin{equation*}
		\iota_S(x)=\begin{cases*}
		0&if $x\in S$,\\
		\infty&else
		\end{cases*}
		\end{equation*}
		\item $\textup{dist}(X,Y)=\inf\{ |x-y|\:|\:x\in X,y\in Y \}$: \textbf{distance} between sets $X,Y\subset\R^n$ (abbreviate $\textup{dist}(x,Y)=\textup{dist}(\{ x \},Y)$)
		\item $f_\varepsilon=f\ast\eta_\varepsilon$: (component-wise) \textbf{$\varepsilon$-mollification} (if well-defined) of function $f:\R^n\to\R^k$ with $\eta_\varepsilon(x)=\varepsilon^{-n}\eta(\varepsilon^{-1}x)$, where $\eta$ positive and radially symmetric mollifier on $\R^n$ with support in $B_1(0)$
		\item $f_i\rightrightarrows f$: \textbf{uniform convergence} of sequence of functions $f_i:D\to V$ ($V$ a normed vector space) to function $f$
		\item $|f|_\infty$: \textbf{supremum norm} of $f:D\to V$ with  $(V,\|.\|)$ a normed vector space, defined by $|f|_\infty=\sup_D\|f\|$
		\item $C^{0,1}(D;\R^k),C_c^\infty(\R^n;\R^k)$: $\R^k$-valued \textbf{Lipschitz functions} on $D\subset\R^n$, smooth and compactly supported \textbf{test functions} $\R^n\to\R^k$
		\item $\mathcal{B}(\cyan{\mathcal{T}})$: \textbf{Borel $\sigma$-algebra} of topological space $\cyan{\mathcal{T}}$
		\item \cyan{$\mu\mres A$: \textbf{restriction} to $A\in\mathcal{A}$ of map $\mu:\mathcal{A}\to X$ from $\sigma$-algebra $\mathcal{A}$ to some set $X$ (e.g.\:outer or Radon measure),
		\begin{equation*}
		(\mu\mres A)(B)=\mu(A\cap B)\qquad\textup{for all }B\in\mathcal{A}
		\end{equation*}}
		\item $\mathcal{L}^k,\mathcal{H}^k$: $k$-dimensional \textbf{Lebesgue} and \textbf{Hausdorff measure}
		\item $|\cyan{F}|$: \textbf{total variation measure} of \cyan{Radon measure $F:\mathcal{B}(\R^n)\to\R^k$}
		\item \cyan{$\frac{\e F}{\e \nu }$: \textbf{Radon\textendash Nikodym derivative} of Radon measure $F:\mathcal{B}(\R^n)\to\R^k$ with respect to measure $\nu:\mathcal{B}(\R^n)\to[0,\infty)$ (if $F$ is absolutely continuous with respect to $\nu$)}
		\item $|F|_H$: \textbf{total variation with respect to norm $H$} on $\R^k$ of Radon measure $F:\mathcal{B}(\R^n)\to\R^k$,
		\begin{equation*}
		|F|_H=\sup\bigg\{\sum_{i}H(F(B_i))\:\bigg|\:(B_i)\subset \mathcal{B}(\R^n)\textup{ countable partition of }\R^n\bigg\}
		=\int H\left( \frac{\e F}{\e |F|} \right)\,\mathrm{d}|F|
		\end{equation*}
		\item $\ws$: convergence in \textbf{weak-$*$ topology} (equivalent to pointwise convergence for distributions)
		\item $\partial f(x)$: \textbf{subdifferential} of convex $f:\R^k\to\R$ at $x\in\R^k$,
		\begin{equation*}
		\partial f(x)=\left\{ g\in\R^k\:\middle|\: f(y)\geq f(x)+g\cdot (y-x)\textup{ for all }y\in\R^k \right\}
		\end{equation*}
		\item $V^*,f^*,\langle v,v^*\rangle$: \textbf{topological dual space} of normed vector space $V$ and \textbf{convex conjugate} $f^*:V^*\to\R\cup\{-\infty,\infty\} $ of $f:V\to\R\cup\{\infty \}$, 
		\begin{equation*}
		f^*(v^*)=\sup_{v\in V}\:\langle v,v^*\rangle-f(v)
		\end{equation*}
		for $v^*\in V^*$, where $\langle v,v^*\rangle$ \textbf{dual pairing}
		\item $\text{dom}(f)$: \textbf{effective domain} $\{ x\in V\,|\, f(x)<\infty \}$ of $f:V\to\R\cup\{\infty \}$	
		\item $A^\dagger$: \textbf{adjoint} of linear operator $A$ 
		\item $\textup{aff}(S)$: \textbf{affine hull} of $S\subset\R^n$
\end{itemize}
\subsection{Currents and flat chains with coefficients in $\R^m$}
\label{currentschains}
Let $k\in\mathbb{N}_0$ and $m\in\mathbb{N}$. Moreover, assume that $h:\R^m\to[0,\infty)$ is a norm. In this section, we define $\mathcal{D}_k^m(\R^n)$, the space of $k$-currents in $\R^n$ with coefficients in $\R^m$, as well as relevant functionals and subsets. Initially, we follow the approach in \cite{MMSR16} and set $\mathcal{D}_k^m(\R^n)=\mathcal{D}_k^1(\R^n)^m$, where $\mathcal{D}_k^1(\R^n)$ is the space of (classical) real $k$-currents as in \cite{F,FF60}. We slightly generalize the definitions incorporating the norm $h$ (yielding, inter alia, the $h$-mass functional $\mathcal{M}_h$). The construction in \cite{MMSR16} corresponds to the case $h=|\cdot|$. We will also recall some classical statements about currents and flat chains which we will use in this article.
\begin{itemize}
\item $\Lambda_k\R^n$: \textbf{$k$-vectors} in $\R^n$ 
\item $\alpha_1\wedge\ldots\wedge\alpha_k$: \textbf{simple $k$-vector} in $\Lambda_k\R^n$, where $\alpha_i\in\R^n$ and $\wedge$ \textbf{exterior multiplication}
\item $|\cdot|_0$: \textbf{mass} on $\Lambda_k\R^n$ defined by 
\begin{equation*}
|\alpha|_0=\inf\left\{\sum_{i=1}^N|\alpha_i|_0\:\middle|\:\alpha=\sum_{i=1}^N\alpha_i,\alpha_i\in\Lambda_k\R^n\:\textup{simple}\right\},
\end{equation*}
where $|\beta|_0=\sqrt{\textup{det}(\beta_i\cdot\beta_j)_{ij}}$ for simple $\beta=\beta_1\wedge\ldots\wedge\beta_k\in\Lambda_k\R^n$ is the volume of the parallelotope spanned by $\beta_1,\ldots,\beta_k$
\item $\Lambda^k\R^n$: \textbf{$k$-covectors} in $\R^n$, i.e.\:linear forms $L:\Lambda_k\R^n\to\R$ 
\item $\Lambda_m^k\R^n$: $k$-covectors in $\R^n$ \textbf{with coefficients} in $\R^m$, i.e.\:bilinear maps $B:\R^m\times\Lambda_k\R^n\to\R$ with \textbf{components} $B_i\in\Lambda^k\R^n$ defined by
\begin{equation*}
B_i(\alpha)=B(b_i,\alpha),
\end{equation*}
where $b_1,\ldots,b_m$ is the standard basis of $\R^m$
\item $|\cdot|_h$: \textbf{$h$-comass} on $\Lambda_m^k\R^n$,
\begin{align*}
|B|_h&=\sup\left\{h_*(B(\cdot,\alpha))\:\middle|\:\alpha\in\Lambda_k\R^n\:\textup{simple},|\alpha|_0\leq 1\right\}\\
&=\sup\left\{ B(\theta,\alpha) \:\middle|\:\alpha\in\Lambda_k\R^n\:\textup{simple},|\alpha|_0\leq 1,\theta\in\R^m,h(\theta)\leq1\right\}
\end{align*}
\item $\Omega_{c}^k(\R^n)$: \textbf{compactly supported smooth differential $k$-forms} $\omega:\R^n\to\Lambda^k\R^n$, i.e.\:if $\omega=\sum_If_I\e p_I$ (where $f_I:\R^n\to\R$, $\e p_I=\e p_{i_1}\wedge\ldots\wedge \e p_{i_k}$, $I\in\{ (i_1,\ldots,i_k)\:|\:1\leq i_1<\ldots<i_k\leq n \}$, $p_i(x)=x_i$, $\e$ total differential, \cyan{$\e p_{i_1}\wedge\ldots\wedge \e p_{i_k}$ denoting the exterior product of the $\e p_{i_j}$}), then $f_I\in C_c^\infty(\R^n)$ 
\item $\Omega_{m,c}^k(\R^n)$: compactly supported and smooth differential $k$-forms $\omega:\R^n\to\Lambda_m^k\R^n$, i.e.\:\textbf{components} $x\mapsto\omega_i(x)=(\omega(x))_i$ lie in $\Omega_{c}^k(\R^n)$ for all $i=1,\ldots,m$
\item $\cyan{\e\omega}$: \textbf{exterior differential} $\e\omega\in\Omega_{m,c}^{k+1}(\R^n)$ of $\omega\in\Omega_{m,c}^k(\R^n)$,
\begin{equation*}
(\e\omega)_i=\e\omega_i
\end{equation*}
\item $|\cdot|_{h}^\infty$: \textbf{$h$-comass} on $\Omega_{m,c}^k(\R^n)$,
\begin{equation*}
|\omega|_{h}^\infty=\sup_{x\in\R^n}|\omega(x)|_{h}
\end{equation*}
\item $|\cdot|_{h,K}^\infty$: $h$-comass on $\Omega_{m,c}^k(\R^n)$ \textbf{with respect to compact $K\subset\R^n$},
\begin{equation*}
|\omega|_{h,K}^\infty=\sup_{x\in K}|\omega(x)|_{h}
\end{equation*}
\item $\mathcal{D}_k^m(\R^n)$: \textbf{$k$-currents in $\R^n$ with coefficients in $\R^m$}, i.e.\:linear functionals $T:\Omega_{m,c}^k(\R^n)\to\R$ with \textbf{components} $T_i:\Omega_c^k(\R^n)\to\R$, $T_i(\omega)=T(\omega^i)$ (with $\omega^i$ defined by $\omega^i_j=\omega$ if $j=i$ and $\omega^i_j=0$ else) that are continuous (in the distributional sense with respect to coefficients of type $f_I$ as above)
\item $\partial T$: \textbf{boundary} of $T\in\mathcal{D}_k^m(\R^n)$,
\begin{equation*}
\partial T(\omega)=T(\e\omega)
\qquad\textup{for all $\omega\in\Omega_{m,c}^{k-1}(\R^n)$}
\end{equation*}
\item $\mathcal{M}_h$: \textbf{$h$-mass} on $\mathcal{D}_k^m(\R^n)$, 
\begin{equation*}
\mathcal{M}_h(T)=\sup_{|\omega|_h^\infty\leq 1}T(\omega)
\qquad\textup{for all $T\in\mathcal{D}_k^m(\R^n)$}
\end{equation*}
\item $\mathcal{M}$: $h$-mass for $h=|\cdot|$ the Euclidean norm; note that all $h$-masses are norm-equivalent
\item $\mathcal{M}_k^m(\R^n)$: $k$-currents in $\R^n$ with coefficients in $\R^m$ and \textbf{finite mass} (also termed \textit{representable by integration}),
\begin{equation*}
\mathcal{M}_k^m(\R^n)=\mathcal{D}_k^m(\R^n)\cap\{ \mathcal{M}_h<\infty \}
=\mathcal{D}_k^m(\R^n)\cap\{ \mathcal{M}<\infty \}
\end{equation*}
\item $\textup{supp}(T)$: \textbf{support} of $T\in\mathcal{D}_k^m(\R^n)$,
\begin{equation*}
\textup{supp}(T)=\R^n\backslash U,
\end{equation*}
where $U\subset\R^n$ largest open set with $T(\omega)=0$ for all $\omega\in \Omega_{m,c}^k(\R^n)$ with support in $U$
\end{itemize}
\begin{memo}[$k$-currents with finite mass and \textbf{Radon measures}]
	\label{currentmeasure}
	If $T\in\mathcal{M}_0^m (\R^n)$, then each component $T_i:C_c^\infty (\R^n)\to\R$ is a distribution of order $0$. Hence, by the Riesz representation theorem it can be identified with a Radon measure $\mu_i:\mathcal{B}(\R^n)\to\R$. Therefore, the current $T$ can be seen as a vector-valued measure $\cyan{\vec{\mu}}_0(T):\mathcal{B}(\R^n)\to\R^m$ with components $\cyan{\vec{\mu}}_0(T)_i=\mu_i$ for all $i=1,\ldots,m$. 
	Similarly, every $T\in\mathcal{M}_1^m (\R^n)$ can be viewed as a matrix-valued measure $\cyan{\vec{\mu}}_1(T):\mathcal{B}(\R^n)\to\R^{m\times n}$; each component $T_i:\Omega_c^1(\R^n)=C_c^\infty(\R^n;\R^n)\to\R$ is a classical $1$-current with finite mass, thus it can be viewed as a vector-valued measure $\mathcal{F}_i:\mathcal{B}(\R^n)\to\R^n$ via $T_i(\phi)=\int\phi\cdot\mathrm{d}\mathcal{F}_i$ for all $\phi\in C_c^\infty(\R^n;\R^n)$. We identify $T$ with the \red{$\R^{m\times n}$-valued} measure $\cyan{\vec{\mu}}_1(T)$ with row $i$ being equal to $\mathcal{F}_i$ for all $i=1,\ldots,m$. In general, any $T\in\mathcal{M}_k^m (\R^n)$ can be represented as a Radon measure $\cyan{\vec{\mu}}_k(T):\mathcal{B}(\R^n)\to\R^m\otimes\Lambda_k\R^n$ (cf.\:\cite[Sec.\:4.4]{MMSR16}). If $m=1$, then (using the dual representation of the total variation and the density of smooth in continuous functions) $\mathcal{M}(T)=|\cyan{\vec{\mu}}_k(T)|(\R^n)$.
\end{memo}
\begin{itemize}
\item to simplify the notation, we will write $T=\cyan{\vec{\mu}}_k(T)$ for all $T\in\mathcal{M}_k^m (\R^n)$
\item $\mathcal{N}_{k,K}^m(\R^n)$: \textbf{normal $k$-currents} in $\R^n$ with coefficients in $\R^m$ and \textbf{support in compact} $K\subset\R^n$, 
\begin{equation*}
\mathcal{N}_{k,K}^m(\R^n)=\mathcal{D}_{k}^m(\R^n)\cap\{ \mathcal{M}_h+\mathcal{M}_h\circ\partial <\infty \}\cap\{ \textup{supp}\subset K \}
\end{equation*}
\item $\mathcal{N}_{k}^m(\R^n)$: \textbf{normal $k$-currents} in $\R^n$ with coefficients in $\R^m$, 
\begin{equation*}
\mathcal{N}_{k}^m(\R^n)=\bigcup_{\substack{K\subset\R^n\\\textup{compact}}}\mathcal{N}_{k,K}^m(\R^n)
\end{equation*}
\end{itemize}
\begin{memo}[Normal $\{ 0,1 \}$-currents and \textbf{divergence measure vectorfields}]
By \cref{currentmeasure} $\mathcal{N}_0^m(\R^n)$ coincides with the space of compactly supported Radon measures $\mathcal{B}(\R^n)\to\R^m$. Moreover, the space $\mathcal{N}_1^m(\R^n)$ equals the space of $m$-tuples of compactly supported divergence measure vectorfields, i.e.\:vector-valued Radon measures whose distributional divergence is a (scalar) Radon measure \cyan{\cite[Sec.\:5]{Sil}}.
\end{memo}
Next, we follow \cite{F} and define flat $1$-chains with coefficients in $\R^m$ via completions of the spaces $\mathcal{N}_{k,K}^m(\R^n)$ using dual flat $h$-seminorms.
\begin{itemize}
\item $|\omega|_{h,K}^\flat$: \textbf{flat $h$-seminorm} of $\omega\in\Omega_{m,c}^k(\R^n)$ \textbf{with respect to compact $K\subset\R^n$},
\begin{equation*}
|\omega|_{h,K}^\flat=\max\{ |\omega|_{h,K}^\infty,|\e \omega|_{h,K}^\infty \}
\end{equation*}
 \item $|T|_{h,K}^\flat$: \textbf{dual flat $h$-seminorm} of $T\in\mathcal{D}_k^m(\R^n)$ with respect to compact $K\subset\R^n$,
 \begin{equation*}
 |T|_{h,K}^\flat=\sup\{  T(\omega)\:|\:\omega\in\Omega_{m,c}^k(\R^n), |\omega|_{h,K}^\flat\leq 1 \}
 \end{equation*}
 \begin{memo}[Boundary has smaller flat $h$-seminorm]
 \label{boundflat}
\cyan{Note that
 \begin{equation*}
  |\partial T|_{h,K}^\flat=\sup\{  T(\e\omega)\:|\:\omega\in\Omega_{m,c}^k(\R^n), |\omega|_{h,K}^\flat\leq 1 \}\leq | T|_{h,K}^\flat
 \end{equation*}
 because $|\e\omega|_{h,K}^\flat=\max\{ |\e\omega|_{h,K}^\infty,|\e^2 \omega|_{h,K}^\infty \}=|\e\omega|_{h,K}^\infty\leq |\omega|_{h,K}^\flat\leq 1$ for every $\omega$ in the supremum.}
 \end{memo}
 \item $|\cdot|_{K}^\flat$: flat $h$-seminorm with respect to compact $K\subset\R^n$ and with $h=|\cdot|$ the Euclidean norm
\item $\mathcal{F}_{k,K}^m(\R^n)$: \textbf{flat $k$-chains} in $\R^n$ with coefficients in $\R^m$ and \textbf{support in compact} $K\subset\R^n$,
\begin{equation*}
\mathcal{F}_{k,K}^m(\R^n)=\overline{\mathcal{N}_{k,K}^m(\R^n)}^{|\cdot|_{h,K}^\flat},
\end{equation*}
i.e.\:the closure of $\mathcal{N}_{k,K}^m(\R^n)$ in $\mathcal{D}_{k}^m(\R^n)$ with respect to $|\cdot|_{h,K}^\flat$
\item $\mathcal{F\!\!M}_{k,K}^m(\R^n)$: flat $k$-chains in $\R^n$ with coefficients in $\R^m$, support in compact $K\subset\R^n$, and \textbf{finite mass},
\begin{equation*}
\mathcal{F\!\!M}_{k,K}^m(\R^n)=\mathcal{F}_{k,K}^m(\R^n)\cap\mathcal{M}_{k,K}^m(\R^n)
\end{equation*}
\item $\mathcal{F}_{k}^m(\R^n)$: \textbf{flat $k$-chains} in $\R^n$ with coefficients in $\R^m$,
\begin{equation*}
\mathcal{F}_{k}^m(\R^n)=\bigcup_{\substack{K\subset\R^n\\\textup{compact}}}\mathcal{F}_{k,K}^m(\R^n)
\end{equation*}
\item $\mathcal{F\!\!M}_{k}^m(\R^n)$: flat $k$-chains in $\R^n$ with coefficients in $\R^m$ and \textbf{finite mass},
\begin{equation*}
\mathcal{F\!\!M}_{k}^m(\R^n)=\mathcal{F}_{k}^m(\R^n)\cap\mathcal{M}_{k}^m(\R^n)
\end{equation*}
\end{itemize}
To ensure that we can apply classical theory, we also recall the alternative definition (resulting in a space $\mathbb{F}_k^m(\R^n)$) which was given by Fleming \cite{F66} using Whitney's textbook \cite{W57}. The (non-trivial) equality $\mathcal{F}_k^m(\R^n)=\mathbb{F}_k^m(\R^n)$ was elaborated by Federer \cite[Thm.\:4.1.23 \&\:Lem.\:4.2.23]{F}.
\begin{itemize}
\item $\mathbb{P}_{k}^m(\R^n)$: \textbf{polyhedral $k$-chains} in $\R^n$ with coefficients in $\R^m$, i.e.\:finite sums $\sum_s\theta_s\otimes\vec{s}\mathcal{H}^k\mres s$\todo[disable]{$\mres$ not introduced} of $k$-simplices $s$ with unit \textbf{orientation} $\vec{s}=s_1\wedge\ldots\wedge s_k\in\Lambda_k\R^n$ (i.e.\:$\vec{s}$ is tangential to the relative interior of $s$ and satisfies $|\vec{s}|_0=1$) and \textbf{multiplicity} $\theta_s\in\R^m$\red{;
note $\mathbb{P}_{k}^m(\R^n)\subset\mathcal{D}_k^m(\R^n)$ as well as $\partial\sum_s\theta_s\otimes\vec{s}\mathcal{H}^k\mres s=\sum_s\theta_s\otimes\vec{s}\mathcal{H}^{k-1}\mres\partial s$
with $\partial s$ the oriented boundary of simplex $s$}
\item $\mathbb{M}_h(P)$: \textbf{$h$-mass} of $P=\sum_s\theta_s\otimes\vec{s}\mathcal{H}^k\mres s\in\mathbb{P}_{k}^m(\R^n)$,
\begin{equation*}
\mathbb{M}_h(P)=\sum_sh(\theta_s)\mathcal{H}^k(s)
\end{equation*}
\item $\mathbb{F}_h(P)$: \textbf{flat $h$-norm} of $P\in\mathbb{P}_{k}^m(\R^n)$,
\begin{equation*}
\mathbb{F}_h(P)=\inf\left\{\mathbb{M}_h(Q)+\mathbb{M}_h(\partial Q-P) \:\middle|\: Q\in\mathbb{P}_{k+1}^m(\R^n)\right\}
\end{equation*}
\item $\mathbb{F}$: flat $h$-norm for $h=|\cdot|$ the Euclidean norm; note that all \cyan{flat} $h$-norms are norm-equivalent
\item $\mathbb{P}_{k,K}^m(\R^n)$: polyhedral $k$-chains with coefficients in $\R^m$ and \textbf{support in compact and convex} $K\subset\R^n$
\end{itemize}
\begin{itemize}
\item $\mathbb{F}_{k,K}^m(\R^n)$: \textbf{flat $k$-chains} in $\R^n$ with coefficients in $\R^m$ and \textbf{support in compact and convex} $K\subset\R^n$,
\begin{equation*}
\mathbb{F}_{k,K}^m(\R^n)=\overline{ \mathbb{P}_{k,K}^m(\R^n)}^{\mathbb{F}_h},
\end{equation*}
i.e.\:the completion of $\mathbb{P}_{k,K}^m(\R^n)$ with respect to the norm $\mathbb{F}_h$ \cyan{in $\mathcal{D}_k^m(\R^n)$}
\item convergence with respect to $\mathbb{F}_h$ is called \textbf{flat convergence}
\item $\mathbb{F}_{k}^m(\R^n)$: \textbf{flat $k$-chains} in $\R^n$ with coefficients in $\R^m$,
\begin{equation*}
\mathbb{F}_{k}^m(\R^n)=\bigcup_{\substack{K\subset\R^n\\\textup{compact \& convex}}}\mathbb{F}_{k,K}^m(\R^n)
\end{equation*}
\item $\mathbb{M}_h(F)$: \textbf{$h$-mass} of $F\in\mathbb{F}_{k}^m(\R^n)$,
\red{defined as relaxation of the $h$-mass for polyhedral chains, i.e.}
\begin{equation*}
\mathbb{M}_h(F)=\inf\left\{ \liminf_j\:\mathbb{M}_h(P_j)\:\middle|\: (P_j)\subset\mathbb{P}_k^m(\R^n),\mathbb{F}_h(P_j-F)\to 0  \right\}
\end{equation*}
\item $\mathbb{M}$: $h$-mass for $h=|\cdot|$ the Euclidean norm; note that all $h$-masses are norm-equivalent
\item $\mathbb{N}_h$: \textbf{sum of mass and boundary mass} $\mathbb{N}_h:\mathbb{F}_k^m(\R^n)\to [0,\infty]$,\todo[disable]{here $\partial$ is used for flat chains}
\begin{equation*}
\mathbb{N}_h=\mathbb{M}_h+\mathbb{M}_h\circ\partial
\end{equation*}
\end{itemize}
\begin{memo}[Flat convergence]
\label{polysemi}
If $P\in\mathbb{P}_{k,K}^1(\R^n)$ for compact and convex $K\subset\R^n$, then $\mathbb{F}(P)=|P|_{K_\delta}^\flat$ for all $\delta>0$ by \cite[p.\:382 \&\:Lem.\:4.2.23]{F}, where $K_\delta=\{ \textup{dist}(\cdot,K)\leq\delta \}$. Therefore, flat convergence $\mathbb{F}(\cyan{F^j})\to 0$ of $(\cyan{F^j})\subset\mathbb{F}_k^1(\R^n)$ inside some bounded domain implies $|\cyan{F^j}|_{K}^\flat\to 0$ for sufficiently large compact $K\subset\R^n$. On the other hand, if $|\cyan{F^j}|_{K}^\flat\to 0$ inside some compact and convex $K\subset\R^n$, then $|\cyan{F^j}|_K^\flat=|\cyan{F^j}|_{K_\delta}^\flat$, thus $\mathbb{F}(\cyan{F^j})\to 0$. This extends to flat chains with coefficients in $\R^m$ since, using norm-equivalence, we have
\begin{equation}
C_1\sum_{i=1}^{m}|F_i|_K^\flat\leq|F|_{h,K}^\flat\leq C_2\sum_{i=1}^{m}|F_i|_K^\flat\label{flatnormequi}
\end{equation} 
for all $F\in\mathbb{F}_{k,K}^m(\R^n)$ and compact $K\subset\R^n$ (with some constants $C_1,C_2>0$ depending on $h$).\todo[disable]{notation clash: sequence $F_i$ vs. components}
\end{memo}
\begin{memo}[$\mathbb{F}_{k}^m(\R^n)=\mathcal{F}_{k}^m(\R^n)$]
\label{equiflat}
By \cref{polysemi} we have $\mathbb{F}_{k}^1(\R^n)\subset\mathcal{F}_{k}^1(\R^n)$. On the other hand, for every $F\in \mathcal{F}_{k}^1(\R^n)$ \cite[Thm.\:4.1.23]{F} implies the existence of some compact and convex $K\subset\R^n$ and a sequence of polyhedral $k$-chains $(P_j)\subset\mathbb{P}_{k,K}^1(\R^n)$ with $|P_j-F|_K^\flat\to 0$, thus $F\in \mathbb{F}_{k}^1(\R^n)$. Therefore, we have $\mathbb{F}_{k}^1(\R^n)=\mathcal{F}_{k}^1(\R^n)$ \cite[p.\:341]{F}. Using canonical isomorphisms and relation \labelcref{flatnormequi}, it is straightforward to check that
\begin{equation*}
\mathbb{F}_{k}^m(\R^n)=(\mathbb{F}_{k}^1(\R^n))^m=(\mathcal{F}_{k}^1(\R^n))^m=\mathcal{F}_{k}^m(\R^n).
\end{equation*}
\end{memo}
\begin{memo}[Mass equivalence]
\label{massequiv}
We have $\mathbb{M}\geq\mathcal{M}$ on $\mathcal{F}_k^1(\R^n)$. Indeed, assuming that $\mathbb{M}(F)<\infty$ for $F\in\mathcal{F}_k^1(\R^n)$, there exists a sequence $(P_j)\subset\mathbb{P}_k^1(\R^n)$ with $\mathbb{F}(P_j-F)\to 0$ and $\mathbb{M}(P_j)\to\mathbb{M}(F)$. It is easy to check (applying \cite[p.\:382]{F}) that this implies $P_j\ws F$. Hence, using $\mathbb{M}(P_j)=\mathcal{M}(P_j)$ (which is readily established on polyhedral chains),
\begin{equation*}
\mathbb{M}(F)=\lim_j\:\mathbb{M}(P_j)=\lim_j\mathcal{M}(P_j)\geq \mathcal{M}(F).
\end{equation*}
The reverse (non-trivial) inequality $\mathbb{M}\leq\mathcal{M}$ follows from \cite[Thm.\:4.1.23]{F}. By norm-equivalence and $\mathbb{M}=\mathcal{M}$ on $\mathcal{F}_k^1(\R^n)$, one obtains
\begin{equation}
C_1\sum_{i=1}^{m}\mathbb{M}(F_i)\leq\mathbb{M}_h(F)\leq C_2\sum_{i=1}^{m}\mathbb{M}(F_i)\qquad\textup{and}\qquad\tilde{C}_1\mathcal{M}_h(F)\leq\mathbb{M}_h(F)\leq \tilde{C}_2\mathcal{M}_h(F)\label{massequi}
\end{equation} 
for all $F\in\mathcal{F}_{k}^m(\R^n)$, where constants $C_1,C_2,\tilde{C}_1,\tilde{C}_2>0$ depend on $h$. We will see that one can take $\tilde{C}_1=\tilde{C}_2=1$.
\end{memo}
Finally, we collect further classical statements \cyan{relevant to this article.}
\begin{prop}[{Formula for flat norm, \cite[Thm.\:3.1]{F66} and \cyan{\cite[pp.\:367 \&\:382]{F}}}]
\label{flatform}
For all $F\in\mathcal{F}_{k}^m(\R^n)$ we have 
\begin{equation*}
\mathbb{F}_h(F)=\inf_{F_1,F_2}\:\mathbb{M}_h(F_1)+\mathbb{M}_h(F_2),
\end{equation*}
where the infimum is over $F_1\in\mathcal{F}_{k}^m(\R^n)$ and $F_2\in\mathcal{F}_{k+1}^m(\R^n)$ with $F=F_1+\partial F_2$. \red{Analogously}, for all compact $K\subset\R^n$ and $F\in\mathcal{F}_{k,K}^1(\R^n)$,
\begin{equation*}
\cyan{|F|_K^\flat=\inf_G\mathcal{M}(G)+\mathcal{M}(F-\partial G),}
\end{equation*}
\cyan{where infimum is over $G\in\mathcal{F}_{k+1,K}^1(\R^n)$.}
\end{prop}
\begin{prop}[{Simultaneous approximation \cite[Thm.\:5.6]{F66}}]
\label{simax}
Let $F\in\mathcal{N}_k^m(\R^n)$. There exists a sequence $(P_i)\subset\mathbb{P}_k^m(\R^n)$ with $\mathbb{F}(P_i-F)\to 0$, $\mathbb{M}_h(P_i)\to\mathbb{M}_h(F)$, and $\mathbb{M}_h(\partial P_i)\to\mathbb{M}_h(\partial F)$.
\end{prop}
The following statement uses the structure of the coefficient group $G=\R^m$ (closed balls are compact) and can be derived from the deformation theorem \cite[Thm.\:7.3]{F66}.
\begin{prop}[{Compactness for normal $\cyan{k}$-currents with coefficients in $\R^m$ \cite[Lem.\:7.4 \&\:Cor.\:7.5]{F66}}]
\label{Compactness}
Let $(F_i)\subset\mathcal{N}_{k,K}^m(\R^n)$ with $\sup_i\mathbb{N}_h(F_i)=C<\infty$. Then we have $\mathbb{F}(F_i-F)\to 0$ up to a subsequence for some $F\in\mathcal{N}_{k,K}^m(\R^n)$ with $\mathbb{N}_h(F)\leq C$.
\end{prop}

\subsection{Structure of normal $1$-currents with coefficients in $\R^m$}
To finish the section, we recall two important structure theorems for normal $1$-currents $F$ with coefficients in $\R^m$. In \cite[Thm.\:C]{S} it was shown that $F$ can be represented as a measure on curves (each component of $F$ is a \textit{vector charge} with finite boundary mass). More precisely, there is a decomposition $F=S+T$ (with \textit{solenoid} $S$, that is $\partial S=0$) such that $S$ can be represented by a $\R^m$-valued measure concentrated on oriented curves with prescribed length (arbitrarily chosen) and $T$ is given by a $\R^m$-valued measure on simple oriented curves of finite length. A different view on $F$ is given in \cite[Thm.\:5.5]{Sil}, which states that $F$ (more generally, every measure which is a flat $1$-chain with coefficients in $\R^m$) admits a decomposition similar to the distributional gradient of a BV function, that is $F=F_{\textup{jump}}+F_{\textup{Cantor}}+F_{\textup{cont}}$, where $F_{\textup{jump}}$ is concentrated on a countably $1$-rectifiable set, $F_{\textup{Cantor}}$ is diffuse with respect to the one-dimensional Hausdorff measure $\mathcal{H}^1$ (vanishing on Borel sets with finite $\mathcal{H}^1$-measure) and singular with respect to the Lebesgue measure $\mathcal{L}^n$, and $F_{\textup{cont}}$ is absolutely continuous with respect to $\mathcal{L}^n$.

\section{Duality in multi-material transport}
\label{dualitysec}
Let $H$ be generated by a multi-material transport cost $h:\R^m\to[0,\infty)$. In \cref{basics} we show some basic properties of $H$ and the corresponding total variation functional $|\cdot|_H$. These can also be found in \cite[Ch.\:4]{JL23}. The idea of prescribing a cost on (a class of) rank-one matrices or material vectors and then using convex biconjugation (see functions $\widetilde H$ and $\mathrm{co}(\widetilde H)$ below) also appears in \cite[Sect.\:4.2]{BOO18} (motivated by \cite{CCP12}) and \cite[Sect.\:4.2.1]{OM23}. \Cref{Equivalencesection} will be devoted to the proof of \cref{EquivalenceKR} and some auxiliary statements \purple{(including \cref{eqforopt}, which states that the minimizers of \labelcref{F} form a subset of the minimizers of \labelcref{DM})}. Finally, in \cref{cwJ}\cyan{,} we present some applications of calibrations.
\subsection{Dual representation of $H$ and $|\cdot|_H$}
\label{basics}
We begin by showing $H(\theta\otimes\vec{e})=h(\theta)$ and that $H$ is the largest such norm.
To this end\cyan{,} we introduce the \cyan{non-negative} homogeneous auxiliary function
\begin{equation*}
\cyan{\widetilde H} :\R^{m\times n}\to[0,\infty],\qquad
\cyan{\widetilde H} (M)=\begin{cases}
h(\theta)&\text{if }M=\theta\otimes\vec{e}\text{ for some }\theta\in\R^m,\vec{e}\in\mathcal{S}^{n-1},\\
\infty&\text{else}
\end{cases}
\end{equation*}
and show that $H$ coincides with its lower semicontinuous convex envelope $\cyan{\mathrm{co}(\widetilde H)}$.

\begin{proof}[Proof of \cref{propertygenerated}]
Using Legendre--Fenchel conjugation, we have
\begin{equation*}
\cyan{\widetilde H}^{*}(M)
=\sup_{\cyan{N}\in\R^{m\times n}}\cyan{M:N}-\cyan{\widetilde H}(\cyan{N})
=\sup_{\theta\in\R^m,\vec e\in\mathcal{S}^{n-1}}\theta\cdot M\vec e-h(\theta)
=\sup_{\vec e\in\mathcal{S}^{n-1}}\iota_{\partial h(0)}(M\vec e)
=\iota_{\partial H(0)}(M),
\end{equation*}
\cyan{for all $M\in\R^{m\times n}$,} which immediately implies $\cyan{\mathrm{co}(\widetilde H)}=(\cyan{\widetilde{H}}^*)^*=H$.
Hence, $H$ is the largest \cyan{lower semicontinuous and} convex function with $H(\theta\otimes\vec e)\leq h(\theta)$, and it remains to show equality.
For all $\theta\in\partial h(0)$ and $\vec{e},\vec{u}\in\mathcal{S}^{n-1}$ we observe
$(\theta\otimes \vec{e})\vec{u}=\theta(\vec{e}\cdot\vec{u})\in[-1,1]\theta\subset\partial h(0)$,
therefore $\partial h(0)\otimes\mathcal{S}^{n-1}\subset\partial H(0)$. This finally yields
\begin{equation*}
H(\theta\otimes \vec{e})=\sup_{M\in\partial H(0)}\theta\otimes \vec{e}:M\geq\sup_{\tilde{\theta}\otimes\vec{u}\in\partial h(0)\otimes\mathcal{S}^{n-1}}\theta\otimes \cyan{\vec{e}}:\tilde{\theta}\otimes\vec{u}=\sup_{\tilde{\theta}\in\partial h(0)}\theta\cdot\tilde{\theta}=h(\theta).
\qedhere
\end{equation*}
\end{proof}

That $H$ is determined by its value on rank-one matrices implies that its dual norm $H_*$ is an operator norm.

\begin{prop}[$H_*$ operator norm]
\label{Hstarop}
The dual norm $H_*$ is an operator norm with respect to $|.|$ on $\R^n$ and $h_*$ on $\R^m$, 
\begin{equation*}
H_*(M)=\sup_{\vec{e}\in\mathcal{S}^{n-1}}h_*(M\vec{e})
\qquad\textup{for all $M\in\R^{m\times n}$.}
\end{equation*}
\end{prop}
\begin{proof}
Denote the norm on the right-hand side by $\cyan{G}$, then we can write
\begin{equation*}
\cyan{G}(M)
=\sup_{\vec{e}\in\mathcal{S}^{n-1}, h(\theta)\leq1}M:\theta\otimes \vec{e}
=\sup_{\cyan{\widetilde{H}}(\cyan{N})\leq 1}M:\cyan{N}.
\end{equation*}
The statement now follows from the relation
\begin{equation*}
H_*(M)
=\sup_{H(\cyan{N})\leq1}M:\cyan{N}
=\!\!\!\sup_{\substack{k\in\mathbb{N},\cyan{N}=\cyan{N}_1+\ldots+\cyan{N}_k\\\cyan{\widetilde{H}}(\cyan{N}_1)+\ldots+\cyan{\widetilde{H}}(\cyan{N}_k)\leq1}}\!\!\!M:\cyan{N}
=\!\!\sup_{k\in\mathbb{N},\cyan{N}_1+\ldots+\cyan{N}_k}\tfrac{\cyan{N}_1:M+\ldots+\cyan{N}_k:M}{\cyan{\widetilde{H}}(\cyan{N}_1)+\ldots+\cyan{\widetilde{H}}(\cyan{N}_k)}
=\sup_{\cyan{N}}\tfrac{\cyan{N}:M}{\cyan{\widetilde{H}}(\cyan{N})}
=\sup_{\cyan{\widetilde{H}}(\cyan{N})\leq1}M:\cyan{N},
\end{equation*}
where the second equality follows from $H=\cyan{\mathrm{co}(\widetilde H)}$ and the \cyan{non-negative} homogeneity of $H$ and $\cyan{\widetilde{H}}$,
the third and fifth equality again use the \cyan{non-negative} homogeneity,
and the fourth equality follows from the elementary inequality $(a_1+\ldots+a_k)/(b_1+\ldots+b_k)\leq\max\{a_1/b_1,\ldots,a_k/b_k\}$.
\end{proof}

The following conjugation lemma states the dual representation of the total variation measure with respect to $H$ (it is also applicable to non-generated norms). 
\begin{lem}[$|.|_H$ as convex conjugate]
	\label{glem}
	Let $K\subset\R^n$ compact and define $\iota_{\partial H(0),K}:C(K;\R^{m\times n})\to\{ 0,\infty\}$,
\begin{equation*}
	\iota_{\partial H(0),K}(\Phi )=\begin{cases*}0&if $\Phi(x)\in\partial H(0)$ for all $x\in K$,\\\infty&else.\end{cases*}
\end{equation*}
	Then we have $|F|_H=\iota_{\partial H(0),K}^*(F)$ for all $F\in \mathcal{M}_{1}^m(\R^n)$ with $\textup{supp}(F)\subset K$.
\end{lem}
\begin{proof}
Apply \cite[Lem.\:2.9]{CPSV18} with compact metric space $(X=K,|.|)$ and convex, non-negative homogeneous, and finite function $f=H$. Note that $\textup{dom}(H^*)=\partial H(0)$ and $f$ is independent of $x\in K$. Alternatively, one can use \cite{V}.
\end{proof}

\subsection{Equivalence of problems \labelcref{DM,F}, generalized Kantorovich\textendash Rubinstein formula}
\label{Equivalencesection}
In this section, we apply Fenchel\textendash Rockafellar duality to the problems \labelcref{DM,F}. The application to \labelcref{DM} can also be found in \cite[Thm.\:1.5.3.3]{JL23}. We will see that the dualization of \labelcref{F} is not straightforward and requires some preliminaries. First, we will observe that in problems \labelcref{DM,F} minimizers are attained. \cyan{\red{The proof requires} the following projection lemma, which follows from \cite{FF60,F66}}.
\begin{lem}[Projection of flat $k$-chains onto convex sets]
\label{projlem}
\cyan{Let $K\subset\R^n$ compact and convex. Write $p:\R^n\to K$ for the orthogonal projection onto $K$. Then $p$ induces a natural projection (for every $k\in\mathbb{N}_0$) $p_\#:\mathcal{F}_k^m(\R^n)\to\mathcal{F}_{k,K}^m(\R^n)$ satisfying $\mathbb{M}_h\circ p_\#\leq\mathbb{M}_h$. Moreover, if $F\in\mathcal{F}_k^m(\R^n)$ with $\textup{supp}(\partial F)\subset K$, then $\partial (p_\# (F))=\partial F$. In the case $k=1$, we also have $|\cdot|_H\circ p_\#\leq|\cdot|_H$ on $\mathcal{N}_1^m(\R^n)$.}
\end{lem}
\begin{proof}
\cyan{Let $F\in\mathcal{F}_k^m(\R^n)$ and $(P_i)\subset\mathbb{P}_k^m(\R^n)$ with $\mathbb{F}(P_i-F)\to 0$ and $\mathbb{M}_h(P_i)\to\mathbb{M}_h(F)$. By \cite[(5.2)]{F66} and the explanation below on \cite[p.\:168]{F66} we get $\mathbb{F}(p_\#(P_i)-p_\#(F))\to 0$. Moreover, we have $\mathbb{M}_h(p_\#(P_i))\leq\mathbb{M}_h(P_i)$ by \cite[(5.1)]{F66} using that $p$ is $1$-Lipschitz. Therefore,
\begin{equation*}
\mathbb{M}_h(p_\#(F))\leq\liminf_i\:\mathbb{M}_h(p_\#(P_i))\leq\liminf_i\:\mathbb{M}_h(P_i)=\mathbb{M}_h(F).
\end{equation*}
To prove the second statement, we assume that $\textup{supp}(\partial F)\subset K$. By the formula below \cite[(5.1)]{F66} we obtain $\partial (p_\# (P_i))=p_\#(\partial P_i)$. Using this, the continuity properties of $\partial$ and $p_\#$ (see above), and $p_\#(\partial F)=\partial F$ ($p$ restricted to $K$ is the identity) yields
\begin{equation*}
\partial (p_\#(F))=\lim_i\partial (p_\#(P_i))=\lim_ip_\#(\partial P_i)=p_\#(\partial F)=\partial F.
\end{equation*}
It remains to show $|p_\#(F)|_H\leq|F|_H$ for all $F\in \mathcal{N}_1^m(\R^n)$. To this end, we use the definition of $p_\#$ given in \cite[Def.\:3.5]{FF60} (it uses smooth approximations of $p$ and the definition in \cite[Sect.\:2.5]{FF60}), which is also applied in \cite{F66}. Using in this order \cref{glem} (by density we can actually assume $\Phi\in C^\infty(K;\R^{m\times n})$ with $\Phi\in\partial H(0)$ in $K$), mollifications $p_\varepsilon$ of $p$ (which converge uniformly to $p$ on compact sets and whose Lipschitz constant is also bounded by one) and \cite[Def.\:3.5]{FF60}, definition of the pushforward $(p_\varepsilon)_\#(F)$ \cite[Sect.\:2.5]{FF60}, and $\Phi\mathrm{D}p_\varepsilon\in\partial H(0)$ (\cref{generated}) plus \cref{glem}, we obtain
\begin{equation*}
|p_\#(F)|_H=\sup_{\Phi}\int_K\Phi:\mathrm{d}p_\#(F)=\sup_\Phi\lim_{\varepsilon}\int_K\Phi:\mathrm{d}(p_\varepsilon)_\#(F)=\sup_\Phi\lim_{\varepsilon}\int\Phi (p_\varepsilon)\mathrm{D}p_\varepsilon:\mathrm{d}F\leq|F|_H.\qedhere
\end{equation*}
}
\end{proof}
\begin{prop}[Well-posedness of \labelcref{DM,F}]
\label{exmin}
\labelcref{DM,F} admit minimizers.
\end{prop}
\begin{proof}
By \cyan{\cref{boundaryprop}} there is some $F_1\in\mathcal{N}_1^m(\R^n)$ with $\partial F_1=F_0$. Let $K\subset\R^n$ compact and convex with $\textup{supp}(F_1)\subset K$. \cyan{Choose} minimizing sequences $(F_i),(G_i)\subset\mathcal{N}_{1}^m(\R^n)$ with $\partial F_i=\partial G_i=F_0$ and
\begin{equation*}
|F_i|_H\to\inf_{ \partial F=F_0}\:|F|_H\qquad\text{and}\qquad\mathbb{M}_h(G_i)\to\inf_{ \partial G=F_0}\:\mathbb{M}_h(G). 
\end{equation*}
\cyan{Without loss of generality, we can assume $(F_i),(G_i)\subset\mathcal{N}_{1,K}^m(\R^n)$ by \cref{projlem}.}
\todo[inline,disable]{before applying Banach--Alaoglu we need to state/show (probably in an extra lemma?) that one may restrict the support along the minimizing sequence to $K$ (e.g. by projection onto $K$)}
Using the Banach\textendash Alaoglu theorem and \cref{Compactness} we have $F_j\ws F$ and $\mathbb{F}(G_j-G)\to 0$ (for some $F\in\mathcal{M}_1^m(\R^n)$ and $G\in\mathcal{F}_1^m(\R^n)$) up to a subsequence, again indexed by $j$. By \cref{glem} and weak-$*$ lower semicontinuity of convex conjugates we get $|F|_H\leq \liminf_j|F_j|_H$. Further, we obtain $\partial F=F_0$ using $F_j\ws F$. Also recall that $\mathbb{M}_h$ is lower semicontinuous (it was defined as a lower semicontinuous envelope) and $\partial$ is continuous with respect to \cyan{flat convergence inside some bounded domain} (\cyan{cf.\:\cref{boundflat,polysemi}}\todo[inline,disable]{this is not the right reference; instead, we still have to define $\partial$ in the notation section as the continuous extension using Hahn--Banach}). Thus, we have $\mathbb{M}_h(G)\leq\liminf_j\mathbb{M}_h(G_j)$ and $\partial G=F_0$. This shows that $F$ is a minimizer of \labelcref{DM} and $G$ of \labelcref{F}.
\end{proof}
Since it will be useful later, we state the weak duality for \labelcref{DM} separately.
\begin{prop}[Weak duality]
\label{Wduality}
Let $(F,\phi)$ as in \cref{EquivalenceKR}. Then $\int\phi\cdot\mathrm{d}F_0\leq |F|_H$.
\end{prop}
\begin{proof}
Let $K\subset\R^n$ compact containing $\textup{supp}(F_0)$. Consider an $\varepsilon$-mollification $\phi_\varepsilon\in C^\infty (\R^n;\R^m)$ of $\phi$ (recall \cref{basicnot}). We obtain
\begin{equation*}
\int\phi\cdot\mathrm{d}F_0=\lim_\varepsilon\int\phi_\varepsilon\cdot\mathrm{d}F_0=\lim_\varepsilon\int \J\phi_\varepsilon:\mathrm{d}F
=\lim_\varepsilon\int \J\phi_\varepsilon:\frac{\e F}{\e |F|}\,\mathrm{d}|F|
\leq\int H\left(\frac{\e F}{\e |F|}\right)\,\mathrm{d}|F|
=|F|_H
\end{equation*}
using $\phi_\varepsilon\rightrightarrows\phi$ on $K\supset \textup{supp}(F_0)$, $F_0=\partial F$, and $\J\phi_\varepsilon\in\partial H(0)$ by the convexity of $\partial H(0)$.
\end{proof}
Next, we determine the dual space of $(\mathcal{F}_{0,K}^m(\R^n),|\cdot|_K^\flat )$ (\cref{dualflatzero}). To this end, we need the Kantorovich\textendash Rubinstein norm. We use the definition from \cite[Sec.\:2]{H92}. 
\begin{defin}[Kantorovich\textendash Rubinstein norm]
\label{KRnorm}
Let $K\subset\R^n$ be compact. The \textbf{Kantorovich\textendash Rubinstein norm} of $\mu\in\mathcal{N}_{0,K}^1(\R^n)$ is defined by
\begin{equation*}
|\mu|_{\mathrm{KR}}=\inf\left\{  |\mu_0|_{\mathrm{W1}}+|\mu-\mu_0|(K)\:\middle|\:\mu_0\in\mathcal{N}_{0,K}^1(\R^n),\mu_0(K)=0\right\},
\end{equation*}
where $|\mu_0|_{\mathrm{W1}}$ is the classical Kantorovich\textendash Rubinstein norm (the \cyan{Wasserstein-$1$ distance} between the positive and negative part of $\mu_0$),
\begin{equation*}
|\mu_0|_{\mathrm{W1}}=\inf\left\{  \int_{K\times K}|x-y|\,\mathrm{d}\pi(x,y)\:\middle|\:\pi\in\Pi(\mu_0)\right\}.
\end{equation*}
Here $\Pi(\mu_0)$ denotes the set of admissible \textbf{transport plans}, that is measures $\pi\in\mathcal{N}_{0,K\times K}^1(\R^n\times\R^n )$ with $\pi(K,B)-\pi(B,K)=\mu_0(B)$ for all $B\in\mathcal{B}(K\times K)$.
\end{defin}
We will see that the flat seminorm $|\cdot|_K^\flat$ coincides with the Kantorovich\textendash Rubinstein norm on the space of signed Radon measures with support in compact and convex $K\subset\R^n$. As a consequence, we can apply \cite[Thm.\:0]{H92} and obtain the following statement.
\begin{lem}[Dual of flat $0$-chains with coefficients in $\R^m$]
\label{dualflatzero}
Let $K\subset\R^n$ compact and convex. Then we have
\begin{equation*}
 (\mathcal{F}_{0,K}^m(\R^n),|\cdot|_K^\flat )^*=C^{0,1}(K;\R^m).
\end{equation*}
\end{lem}
\begin{proof}
Let $\mu$ and $\mu_0$ as in \cref{KRnorm}. It is well-known that $|\mu_0|_{\mathrm{W1}}$ admits a Beckmann formulation (it can be obtained by twofold application of Fenchel\textendash Rockafellar duality \cyan{\red{and is a special case of} the Beckmann formulation of the classical Wasserstein-$1$ distance \cite[Thm.\:4.6]{San}}),
\begin{equation*}
|\mu_0|_{\mathrm{W1}}=\inf_\mathcal{F}|\mathcal{F}|(K),
\end{equation*}
where \red{the} infimum is over $\mathcal{F}\in\mathcal{N}_{1,K}^1(\R^n)$ with $\partial\mathcal{F}=-\mu_0$. Hence, we get
\begin{equation*}
|\mu|_{\mathrm{KR}}=\inf_{\mu_0,\mathcal{F}}|\mathcal{F}|(K)+|\mu+\partial\mathcal{F}|(K),
\end{equation*}
where \red{the} infimum is over $\mu_0$ and $\mathcal{F}$ as above. Since always $(\partial\mathcal{F})(K)=0$, we can drop the optimization over $\mu_0$,
\begin{equation*}
|\mu|_{\mathrm{KR}}=\inf_{\mathcal{F}}|\mathcal{F}|(K)+|\mu+\partial\mathcal{F}|(K),
\end{equation*}
where $\mathcal{F}\in\mathcal{N}_{1,K}^1(\R^n)$. We have $|\mathcal{F}|(K)=\mathcal{M}(\mathcal{F})$ and $|\mu+\partial\mathcal{F}|(K)=\mathcal{M}(\mu+\partial\mathcal{F})$ by \cref{currentmeasure}, thus (recall \cref{flatform})\todo[inline,disable]{\cref{flatform} needs to be reformulated s.t. $G$ is the minimizer}
\begin{equation*}
|\mu|_{\mathrm{KR}}=\inf_{\mathcal{F}}\:\mathcal{M}(\mathcal{F})+\mathcal{M}(\mu-\partial\mathcal{F})=|\mu|_K^\flat.
\end{equation*}
Using this and \cite[Thm.\:0]{H92} we obtain $(\mathcal{N}_{0,K}^1(\R^n),|.|_K^\flat )^*=C^{0,1}(K)$. Finally, as $\mathcal{F}_{0,K}^m(\R^n)$ is the completion of $\mathcal{N}_{0,K}^m(\R^n)$, we have (invoking the Hahn\textendash Banach theorem)
\begin{equation*}
\left(\mathcal{F}_{0,K}^m(\R^n)\right)^*=\left(\mathcal{N}_{0,K}^m(\R^n)\right)^*=\left(\mathcal{N}_{0,K}^1(\R^n)^m\right)^*=\left(\mathcal{N}_{0,K}^1(\R^n)^*\right)^m=C^{0,1}(K;\R^m).\qedhere
\end{equation*}
\end{proof}
Next, we rewrite the dual constraint which we will obtain when applying Fenchel\textendash Rockafellar duality to \labelcref{F}.
\begin{lem}[Dual constraint]
	\label{dualconstraint}
	Let $K\subset\R^n$ compact and convex. Then for every $\phi\in C^{0,1}(K;\R^m)$ the constraint (recall \cref{dualflatzero})
	\begin{equation*}
	\langle F,\partial^\dagger\phi\rangle\leq\mathbb{M}_h(F)\qquad\textup{for all}\qquad F\in \mathcal{F}_{1,K}^m(\R^n)
	\end{equation*}
	is equivalent to $\J\phi\in\partial H(0)$ a.e.\:in $K$.
\end{lem}
\begin{proof}
	By continuity of $\partial^\dagger\phi$ and the definition of $\mathbb{M}_h$ (via relaxation of $\mathbb{M}_h$ on $\mathbb{P}_{1}^m(\R^n)$ with respect to flat convergence) it is equivalent to require that the inequality is satisfied for all $P\in\mathbb{P}_{1,K}^m(\R^n)$. Every $P$ can be written as a finite sum of non-overlapping oriented and $\R^m$-weighted line segments. Since $\mathbb{M}_h(L_1+L_2)=\mathbb{M}_h(L_1)+\mathbb{M}_h(L_2)$ for non-overlapping line segments $L_1$ and $L_2$, the constraint can be rewritten as
	\begin{equation*}
	\langle\theta\otimes\vec{e}\mathcal{H}^1\mres e,\partial^\dagger\phi\rangle\leq \mathbb{M}_h(\theta\otimes\vec{e}\mathcal{H}^1\mres e)
	\end{equation*}
	for all $\theta\in\R^m$ and line segments $e=[e_+,e_-]\subset K$ with orientation $\vec{e}=\frac{e_--e_+}{|e_--e_+|}\in\mathcal{S}^{n-1}$.
	 For all such $(\theta,e,\vec{e})$ we have, denoting by $\delta_x$ the Dirac measure in $x$,
	\begin{equation}
	\label{ecalc}
	\langle\theta\otimes\vec{e}\mathcal{H}^1\mres e,\partial^\dagger\phi\rangle=\langle\partial (\theta\otimes\vec{e}\mathcal{H}^1\mres e),\phi\rangle=\langle\theta (\delta_{e_-}-\delta_{e_+}),\phi\rangle=\theta\cdot(\phi(e_-)-\phi(e_+)).
	\end{equation}
	Hence, we have to verify the equivalence
	\begin{equation*}
	\theta\cdot(\phi(e_-)-\phi(e_+))\leq h(\theta)\mathcal{H}^1(e)\textup{ for all }\theta\in\R^m,e=[e_+,e_-]\subset K\quad\Leftrightarrow\quad \J\phi\in\partial H(0)\textup{ a.e.\:in }K.
	\end{equation*}
	
	\noindent``$\Rightarrow$'': Pick any point $x$ in the interior of $K$ such that $\phi$ is totally differentiable in $x$ (a.e.\:point in $K$ satisfies this property by Rademacher's theorem). Taking $e_+=x$ and $e_-=e_++\varepsilon\vec{e}$ we have
	\begin{equation*}
	\theta\cdot(\phi(x+\varepsilon\vec{e})-\phi(x))\leq h(\theta)\varepsilon
	\end{equation*}
	for all $\varepsilon>0$ (sufficiently small) and $\theta\in\R^m,\vec{e}\in\mathcal{S}^{n-1}$. Dividing by $\varepsilon$ and letting $\varepsilon\to 0$ we obtain
	\begin{equation}
	\label{equivalenceabove}
	\theta\cdot \J\phi(x)\vec{e}\leq h(\theta)\:\forall\theta,\vec{e}\quad\Leftrightarrow\quad \J\phi(x)\vec{e}\in\partial h(0)\:\forall\vec{e}\quad\Leftrightarrow\quad \J\phi(x)\in\partial H(0),
	\end{equation}
	where the last equivalence follows from \cref{generated}.

	\noindent``$\Leftarrow$'': Let $\theta\otimes\vec{e}\mathcal{H}^1\mres e\in\mathbb{P}_{1,K}^m(\R^n)$ be arbitrary. We can assume that the intersection of $e=[e_+,e_-]$ with the boundary of $K$ is empty (otherwise approximate in the flat norm by a sequence of line segments with this property). Let $\phi_\varepsilon\in C^\infty (\R^n;\R^m)$ be an $\varepsilon$-mollification (with $\varepsilon<\textup{dist}(e,\partial K)$) of $\phi$ (extended by zero to $\R^n$). We have $\J\phi_\varepsilon\in\partial H(0)$ in $K_\varepsilon=K\cap\{ \textup{dist}(.,\partial K)>\varepsilon \}$ by the convexity of $\partial H(0)$, thus $\theta\cdot \J\phi_\varepsilon\vec{e}\leq h(\theta)$ on $e\subset K_\varepsilon$ (see equivalence \labelcref{equivalenceabove}). In particular, we obtain (fundamental theorem of calculus)
	\begin{equation*}
	\theta\cdot(\phi_\varepsilon (e_-)-\phi_\varepsilon (e_+))=\int_e\theta\cdot \J\phi_\varepsilon\vec{e}\,\mathrm{d}\mathcal{H}^1\leq\int_eh(\theta)\,\mathrm{d}\mathcal{H}^1=h(\theta)\mathcal{H}^1(e).
	\end{equation*}
	Letting $\varepsilon\to 0$ we get the desired inequality from $\phi_\varepsilon\rightrightarrows\phi$ on $e$.
\end{proof}

We are now ready to prove the equivalence of problems \labelcref{DM,F} along with the generalized Kantorovich\textendash Rubinstein formula.
\begin{proof}[Proof of \cref{EquivalenceKR}]
\Cref{exmin} already proves the attainment of the minima.
Further, without loss of generality, we can assume that the minimizations are over $F\in\mathcal{N}_{1,K}^m(\R^n)$ with $\partial F=F_0$ and \cyan{sufficiently large compact and convex $K\subset\R^n$ (\cref{projlem})}. \todo[disable]{still requires a projection lemma}

\underline{$\min_F|F|_H=\sup_\phi \int \phi\cdot\mathrm{d }F_0$:} We apply Fenchel\textendash Rockafellar duality. Consider the Banach spaces
	\begin{align*}
	X=C^1(K;\R^m)\textup{ with norm }|\phi|_X=|\phi|_\infty+|\J\phi|_\infty\qquad\textup{and}\qquad Y=C(K;\R^{m\times n})\textup{ with norm }|\Phi|_Y=|\Phi|_{\infty }.
	\end{align*}
	Moreover, introduce the convex function $f:X\to\R$ and the linear and bounded operator $A:X\to Y$ by
	\begin{equation*}
	f(\phi)=\langle\phi,-F_0\rangle\qquad\textup{and}\qquad A\phi=\J\phi.
	\end{equation*}
	Further, let $g=\iota_{\partial H(0),K}$ (recall \cref{glem}) and note that this function is convex by the convexity of $\partial H(0)$.
	For any sequence $(\Phi_j)\subset Y$ with $\Phi_j\rightrightarrows 0$ we have $\Phi_j(x)\in\partial H(0)$ for all $x\in K$ if $j$ is chosen sufficiently large since zero is an interior point of $\partial H(0)$. Hence, function $g$ is continuous in $0\in Y\cap A\textup{dom}(f)$. Using the second constraint qualification in \cite[Thm.\:4.4.18]{BV} we obtain strong duality, 
	\begin{multline*}
	\sup\left\{ \langle\phi,F_0\rangle\: \middle|\:\phi\in C^1(K;\R^m),\J\phi(x)\in\partial H(0)\textup{ for all }x\in K \right\}
	=-\inf\left\{ f(\phi)+g(A\phi)\:\middle|\:\phi\in X \right\}\\
	=-\sup\left\{ -f^*(A^\dagger F)-g^*(-F)\:\middle|\:F\in Y^*\right\}
	=\inf\left\{ f^*(-A^\dagger F)+g^*(F)\:\middle|\: F\in Y^*\right\}.
	\end{multline*}
	We have $f^*(\mu)=\iota_{\{ -F_0\}}(\mu)$ for all $\mu\in X^*$. Moreover, for $F\in Y^*=\mathcal{M}_{1,K}^m(\R^n)$ and $\phi\in X$ we get
	\begin{equation*}
	\langle\phi,A^\dagger F\rangle=\langle A\phi,F\rangle=\langle \J\phi,F\rangle,
	\end{equation*}	 
	therefore $A^\dagger F=\partial F$.
	Using \cref{glem} we can continue the conversion of the primal problem,
	\begin{equation*}
	\inf\left\{ f^*(-A^\dagger F)+g^*(F)\:\middle|\: F\in Y^*\right\}
	=\inf\left\{ \iota_{\{ -F_0\}}(-\partial F)+|F|_H\:\middle|\:F\in \mathcal{M}_{1,K}^m(\R^n)\right\}.
	\end{equation*}
	The latter expression is equal to problem \labelcref{DM}. Finally, it suffices to show that every $\phi\in C^{0,1}(K;\R^m)$ with $\J\phi\in\partial H(0)$ a.e.\:in $K$ can be extended to $\phi\in C^{0,1}(\R^n;\R^m)$ satisfying the same Lipschitz constraint. Then the equivalence follows from \cref{Wduality} (last inequality),
	\begin{equation*}
	\labelcref{DM}=\sup_{\substack{\phi\in X\\ \J\phi\in\partial H(0)\textup{ in }K}}\langle\phi,F_0\rangle\leq \sup_{\substack{\phi\in C^{0,1}(K;\R^m)\\ \J\phi\in\partial H(0)\textup{\:a.e.\:in }K}}\langle\phi,F_0\rangle\leq \labelcref{DM}.
	\end{equation*}
	Pick an arbitrary $\phi\in C^{0,1}(K;\R^m)$ with $\J\phi\in\partial H(0)$ a.e.\:in $K$. Extend $\phi$ to $\R^n$ via $\phi(x)=\phi(\textup{proj}_K(x))$, where $\textup{proj}_K$ denotes the orthogonal projection onto $K$. By the convexity of $K$ the Lipschitz constant with respect to any norms $|\cdot|_1$ on $\R^n$ and $|\cdot|_2$ on $\R^m$ is preserved. For $|\cdot|_1=|\cdot|$ and $|\cdot|_2=h_*$ it is given by $\esssup_KH_*(\J\phi)$ invoking \cref{Hstarop}. Now use that $M\in\partial H(0)$ is equivalent to $H_*(M)\leq 1$.

	\underline{$\min_F\mathbb{M}_h(F)=\max_\phi \int \phi\cdot\mathrm{d }F_0$:} We apply Fenchel\textendash Rockafellar duality again. This time, we will directly get the desired function space for $\phi$. Define Banach spaces $X=\mathcal{F}_{1,K}^m(\R^n)$ and $Y=\mathcal{F}_{0,K}^m(\R^n)$. Recall that $\partial:X\to Y$ is linear and continuous. Moreover, the maps $\varphi=\mathbb{M}_h:X\to[0,\infty ]$ and $\psi=\iota_{\{ F_0\}}:Y\to\{ 0,\infty \}$ are convex, lower semicontinuous, and proper. Set $E=X\times Y$, $\Phi(x,y)=\varphi(x)+\psi(y)$ for $(x,y)\in X\times Y$, and $M=\{ (F,\partial F)\:|\:F\in X \}$. As in the proof of \cite[Cor.\:2.3]{AB86}, we show that $\bigcup_{\lambda\geq 0}\lambda(\textup{dom}(\Phi)-M)$ is a closed vector space (so that later we can apply \cite[Cor.\:2.2]{AB86}). In particular, we claim that (with a slight abuse of notation, we write $\partial X=\{ \partial F| F\in X \}$)
	\begin{equation*}
\bigcup_{\lambda\geq 0}\lambda(\textup{dom}(\Phi)-M)=X\times\partial X.
\end{equation*}
The right-hand side is obviously a vector space, and it is closed since $\partial$ is continuous. It suffices to prove
\begin{equation}
\textup{dom}(\Phi)-M=X\times\partial X,\label{union}
\end{equation}
because then $\textup{dom}(\Phi)-M=\bigcup_{\lambda\geq 0}\lambda(\textup{dom}(\Phi)-M)$. First, we note that
\begin{equation*}
X\times\partial X=\{ (F-G,-\partial G)\:|\:F,G\in X \}=X\times\{ 0\}-M.
\end{equation*}
Now let $(F,0)-(G,\partial G)\in X\times\partial X$ be arbitrary, where $F,G\in X$. Using \cref{flatform} we have $F=G_1-\partial G_2$, where $G_1\in\mathcal{F\!\!M}_{1,K}^m(\R^n)$ and $G_2\in\mathcal{F\!\!M}_{2,K}^m(\R^n)$. Recall that $F_0=\partial F_1$ for some $F_1\in\mathcal{N}_{1,K}^m(\R^n)$. Using this assumption and $\textup{dom}(\Phi)=\mathcal{F\!\!M}_{1,K}^m(\R^n)\times\{ F_0 \}$, we obtain
\begin{equation*}
(F,0)-(G,\partial G)=(G_1+F_1,F_0)-\left( (F_1,\partial F_1)+(\partial G_2,0)+(G,\partial G) \right)\in \textup{dom}(\Phi)-M,
\end{equation*}
which implies ($F,G$ were arbitrary) $X\times\partial X\subset \textup{dom}(\Phi)-M$. The reverse subset relation follows directly from
\begin{multline*}
\textup{dom}(\Phi)-M=\mathcal{F\!\!M}_{1,K}^m(\R^n)\times\{ F_0 \}-\{ (F,\partial F)\:|\: F\in X \}=\{ (G-F,\partial (F_1-F))\:|\:G\in\mathcal{F\!\!M}_{1,K}^m(\R^n),F\in X \}\\
=\{ (G+(F_1-F),\partial (F_1-F))\:|\:G\in\mathcal{F\!\!M}_{1,K}^m(\R^n),F\in X \}\subset X\times \partial X,
\end{multline*}
which finalizes the proof of \cref{union}. Applying \cite[Cor.\:2.2]{AB86} (as in the proof of \cite[Cor.\:2.3]{AB86}) yields strong duality and existence of a dual maximizer \cite[Eqn.\:2.6]{AB86},
\begin{equation}
\inf_{x\in X}\varphi(x)+\psi(\partial x)=\max_{y^*\in Y^*}-\varphi^*(-\partial^\dagger y^*)-\psi^*(y^*).\label{maximum}
\end{equation}
Finally, we identify the dual problem. By \cref{dualflatzero} we have $Y^*=C^{0,1}(K;\R^m)$. Recalling that $\varphi=\mathbb{M}_h$ we get
\begin{equation*}
-\varphi^*(-\partial^\dagger\phi)=\inf_{F\in X}\mathbb{M}_h(F)+\langle F,\partial^\dagger\phi \rangle=\begin{cases*}
-\infty&if $\mathbb{M}_h(F)+\langle F,\partial^\dagger\phi \rangle<0$ for some $F\in X$,\\
0&else,
\end{cases*}
\end{equation*}
for all $\phi\in C^{0,1}(K;\R^m)$. Since $-\varphi^*(-\partial^\dagger\phi)$ is the first term in the maximum in \cref{maximum}, we obtain the constraint
\begin{equation*}
\langle F,\partial^\dagger\phi\rangle\leq\mathbb{M}_h(F)\qquad\textup{for all}\qquad F\in X=\mathcal{F}_{1,K}^m(\R^n).
\end{equation*}
By \cref{dualconstraint} this is equivalent to $\J\phi\in\partial H(0)$ a.e.\:in $K$. As in the first part of the proof, we can extend $\phi$ to $\R^n$ such that this Lipschitz condition is still satisfied. Finally, by the above and \cref{maximum} we get
\begin{equation*}
\labelcref{F}
=\max_\phi-\iota_{\{ F_0\} }^*(\phi)=\max_\phi\int \phi\cdot\mathrm{d}F_0
\end{equation*}
with $\phi$ having the desired properties.
\end{proof}
\cyan{The following \namecref{eqforopt} will be used in the proof of \crefthmpart{equalitymasses}{item2}.
\begin{coro}[$|F|_H=\mathbb{M}_h(F)$ for optimal $F$]
	\label{eqforopt}
	Let $F\in\mathcal{N}_{1}^m(\R^n)$ be a minimizer of problem \labelcref{F}. Then $|F|_H=\mathbb{M}_h(F)$, thus $F$ also minimizes \labelcref{DM}.
\end{coro}
\begin{proof}
The inequality $|F|_H\leq\mathbb{M}_h(F)$ is not difficult to show. We defer its proof to \cref{viaminimizers} (cf.\:first part in proof of \crefthmpart{equalitymasses}{item2}). The reverse (non-trivial) inequality directly follows from \cref{EquivalenceKR} using that $F$ is a minimizer of \labelcref{F},
\begin{equation*}
|F|_H\geq\labelcref{DM}=\labelcref{F}=\mathbb{M}_h(F).\qedhere
\end{equation*}
\end{proof}}
\subsection{Calibrations as weak Jacobians of dual potentials}
\label{cwJ}
Let $h$ be a multi-material transport cost and $H$ generated by $h$. We recall that $\Phi=\J\phi$ is a calibration for $F$ if and only if the pair $(F,\phi)$ is admissible, that is $\partial F=F_0$ and $\J\phi\in\partial H(0)$ a.e., and $\mathbb{M}_h(F)=\int\phi\cdot\mathrm{d}F_0$. The primal-dual extremality conditions \cyan{related to problem \labelcref{F}} can equivalently be written as 
\begin{center}
\begin{minipage}{0.4\textwidth}
\begin{enumerate}[label=E.\arabic*]
\item \qquad $\partial^\dagger\phi\in\partial\mathbb{M}_h(F)$\qquad and\label{e1}
\item \qquad $\partial F\in\partial (-\iota_{\{ F_0 \}})^*(\phi)$.\label{e2}
\end{enumerate}
\end{minipage}
\end{center}
The first condition is equivalent to the admissibility of $\phi$ and $\mathbb{M}_h(F)=\int\phi\cdot\mathrm{d}F_0$, the second to $\partial F=F_0$. This can be obtained from \cite[Rels.\:(2.5), (8.3), and (8.4)]{R67}\footnote{The setting in \cite[Sec.\:3]{R67} is satisfied with $g=-\psi$ and ``topologically paired
real vector spaces'' $(X,X^*)$ and $(Y,Y^*)$. Indeed, every Banach space is topologically paired with its dual by the definition in \cite[Sec.\:2]{R67} (it is automatically locally convex and Hausdorff).} and \cref{dualconstraint}. 
\cyan{In this section}, we briefly remark on some properties of the generalized Kantorovich potentials $\phi$ and give examples which illustrate that all the characteristics of minimizers of \labelcref{F} are hidden in the definition of a calibration. 
\begin{rem}[Interpretation of dual potential $\phi$]\label{rem:dualPotential}
The following three observations indicate the relevance of the antiderivative associated with a calibration.
Assume that $F=\sum_e\theta_e\otimes\vec{e}\mathcal{H}^1\mres e$ is optimal for \labelcref{F} \cyan{with (finitely) discrete source and sink distribution} (that is $\mathbb{M}_h(F)=\min_{\partial G=\partial F}\mathbb{M}_h(G)$, where the sum is over \cyan{finitely many} non-overlapping edges $e=[e_+,e_-]\subset\R^n$ with orientations $\vec{e}\in\mathcal{S}^{n-1}$ and weights $\theta_e\in\R^m$) and let $\Phi=\J\phi$ be a calibration for $F$ (\cref{excal}).
\begin{itemize}
\item For each $e$ and $\mathcal{H}^1$-a.e.\:$x\in e$,
\begin{equation}\label{eqn:extremalityCalibration}
\text{$\partial_{\vec{e}}\phi(x)$ is extremal in the sense that it lies on the boundary of $\partial h(0)$,}
\end{equation}
where $\partial_{\vec{e}}\phi(x)$ denotes the derivative of $\phi$ at $x$ in direction $\vec{e}$ \cyan{(it exists $\mathcal{H}^1$-a.e.\:on $e$ by Rademacher's theorem)}.
Indeed, by the extremality condition \todo[disable]{I do not see how}  $\mathbb{M}_h(F)=\int\phi\cdot\mathrm{d}F_0$ \cyan{and \cref{ecalc} we obtain
\begin{equation*}
\sum_eh(\theta_e)\mathcal{H}^1(e)=\sum_e\langle\theta_e\otimes\vec{e}\mathcal{H}^1\mres e,\partial^\dagger\phi\rangle=\sum_e\theta_e\cdot (\phi(e_-)-\phi(e_+)).
\end{equation*}
We can assume that the edge $e_\varepsilon=[x,x+\varepsilon\vec{e}]\subset e$ ($\varepsilon>0$ sufficiently small) is appearing in the sum (otherwise refine the set of edges). Now recall that for every edge $e$ we have $\theta_e\cdot (\phi(e_-)-\phi(e_+))\leq h(\theta_e)\mathcal{H}^1(e)$ by the admissibility of $\phi$ (equivalence below \cref{ecalc}). In particular, by the above we must have $\theta_e\cdot (\phi(e_-)-\phi(e_+))=h(\theta_e)\mathcal{H}^1(e)$ for all $e$. Inserting $e=e_\varepsilon$, dividing by $\varepsilon$, and letting $\varepsilon\to 0$ yields 
\begin{equation*}
h(\theta_e)=\theta_e\cdot\partial_{\vec{e}}\phi(x),
\end{equation*}
which implies the statement.}
\item Now assume that $\phi$ is totally differentiable at a junction $x$ between different branches $e$. By the above this implies $h(\theta_e)=\theta_e\cdot\Phi(x)\vec{e}$ for all these branches $e$. Since $h(\theta_e)\geq\theta_e\cdot\Phi(x)\vec{u}$ for any $\vec{u}\in\mathcal{S}^{n-1}$ by the admissibility of $\phi$, we obtain $\theta_e^T \J\phi(x)=h(\theta_e)\vec{e}$. Thus,
\begin{equation*}
\sum_{e_-=x}h(\theta_e)\vec{e}=\left( \sum_{e_-=x}\theta_e^T \right)\J\phi(x)=\left( \sum_{e_+=x}\theta_e^T \right)\J\phi(x)=\sum_{e_+=x}h(\theta_e)\vec{e},
\end{equation*}
where mass conservation was used in the second equation. This equality is the well-known momentum conservation law \cyan{(e.g.\:\cite[p.\:255]{X} or \cite[eq.\:(4.31)]{San})} for inner vertices of optimal transport networks\footnote{Classically, it follows from $\nabla_y\left( \sum_{e_-=x}\theta_e\otimes\frac{y-e_+}{|y-e_+|}\mathcal{H}^1\mres[e_+,y]+ \sum_{e_+=x}\theta_e\otimes\frac{e_--y}{|e_--y|}\mathcal{H}^1\mres[y,e_-] \right)\bigg|_{y=x}=0$ by optimality of $x$.}.
\item In the context of branched transport (if multi-material transport is used as a convexification), one may interpret $\phi$ as a $\R^m$-valued \textit{landscape function}. This function was mathematically defined in \cite{S07}, and its Kantorovich potential role was mentioned on \cite[p.\:162]{S07}. It was introduced to describe and analyze the elevation of the landscape in a river basin (with branched drainage systems) in an equilibrium state (where the erosion effects caused by rain became negligible). Now assume that $\textup{supp}(F)$ is acyclic and path-connected. Pick any $x_0\in\textup{supp}(F)$ and set $Z(x_0)=0$. For $x\in\textup{supp}(F)$ and $[v_0,v_1]\cup\ldots\cup[v_{N-1},v_N]$ being the polygonal chain from $x_0$ to $x$ reduce $\phi$ to the function
\begin{equation*}
Z(x)=\sum_{i=1}^N\|\theta_i\|_{\ell^1}^{-1}\theta_i\cdot (\phi(v_i)-\phi(v_{i-1})),
\end{equation*}
where $\theta_i=\theta_e$ if $(v_{i-1},v_i)\cap e\neq\emptyset$. Here $\|\theta\|_{\ell^1}=\sum_j|\theta_j|$ may be interpreted as the total mass of $\theta\in\R^m$. Then this function coincides with the landscape function given in \cite[Def.\:3.1]{Xi14},
\begin{equation*}
Z(e_-)-Z(e_+)=\frac{h(\theta_e)}{\|\theta_e\|_{\ell^1}}\mathcal{H}^1(e)\quad\textup{and}\quad Z|_e\textup{ is linear}\quad\textup{for all }e.
\end{equation*}
\end{itemize}
\end{rem}

As the next example illustrates, calibrations may be very simple, but potentially discontinuous in contrast to the classical notion of calibrations.

\begin{examp}[{Two sources and sinks \cyan{\cite[Sec.\:4.5]{OM23}}}]\label{exm:twoSourcesSinks}
Let $m=n=2$, $h(\theta)=\max\{|\theta_1|,|\theta_2|,|\theta_1-\theta_2|\}$ and set
\begin{equation*}\textstyle
\vec{e}_1={1\choose0},\
\vec{e}_2=\tfrac12{1\choose\sqrt3},\
\vec{e}_3=\tfrac12{1\choose-\sqrt3},\
\theta^1={1\choose1},\
\theta^2={1\choose0},\
\theta^3={0\choose1},
\end{equation*}
and $p_{\cyan{-1}}=\vec{e}_1$, $p_{\cyan{-2}}=\vec{e}_1+\vec{e}_2$, $p_{\cyan{-3}}=\vec{e}_1+\vec{e}_3$, $p_{\cyan{+i}}=-p_{\cyan{-i}}$, $i=1,2,3$ (cf.\:\cref{fig:twoSourcesSinks} left).
It is readily checked that for the source and sink distribution
$
F_0=\theta^2(\delta_{p_{\cyan{-2}}}-\delta_{p_{\cyan{+2}}})+\theta^3(\delta_{p_{-3}}-\delta_{p_{\cyan{+3}}})
$
an optimizer of \labelcref{DM} is given by
\begin{equation*}
F=\theta^1\otimes\vec{e}_1\mathcal{H}^1\mres[p_{\cyan{+1}},p_{\cyan{-1}}]+\theta^2\otimes\vec{e}_2\mathcal{H}^1\mres[p_{\cyan{+2}},p_{\cyan{+1}}]\cup[p_{\cyan{-1}},p_{\cyan{-2}}]+\theta^3\otimes\vec{e}_3\mathcal{H}^1\mres[p_{\cyan{+3}},p_{\cyan{+1}}]\cup[p_{\cyan{-1}},p_{\cyan{-3}}]
\end{equation*}
since it is calibrated by
\begin{equation*}
\Phi(x)=\tfrac12\left(\begin{smallmatrix}1&\sqrt3\\1&-\sqrt3\end{smallmatrix}\right).
\end{equation*}
Likewise, for
$
F_0=\theta^2(\delta_{p_{\cyan{-2}}}-\delta_{p_{\cyan{+3}}})+\theta^3(\delta_{p_{\cyan{-3}}}-\delta_{p_{\cyan{+2}}})
$
two different optimizers of \labelcref{DM} are given by
\begin{gather*}
F=\theta^1\otimes\vec{e}_1\mathcal{H}^1\mres[p_{\cyan{+1}},p_{\cyan{-1}}]+\theta^2\otimes\vec{e}_{\cyan{3}}\mathcal{H}^1\mres[p_{\cyan{+3}},p_{\cyan{+1}}]\cyan{+\theta^2\otimes\vec{e}_2\mathcal{H}^1\mres[p_{\cyan{-1}},p_{\cyan{-2}}]}\\
+\theta^3\otimes\vec{e}_{\cyan{2}}\mathcal{H}^1\mres[p_{\cyan{+2}},p_{\cyan{+1}}]\cyan{+\theta^3\otimes\vec{e}_{\cyan{3}}\mathcal{H}^1\mres[p_{\cyan{-1}},p_{\cyan{-3}}]}\\
\qquad\textup{and}\qquad G=\theta^2\otimes\vec{e}_1\mathcal{H}^1\mres[p_{\cyan{+3}},p_{\cyan{-2}}]+\theta^3\otimes\vec{e}_1\mathcal{H}^1\mres[p_{\cyan{+2}},p_{\cyan{-3}}]
\end{gather*}
since, abbreviating $A_\pm=\{\pm x_1>\sqrt3|x_2|\}$, $B_\pm=\{\pm\sqrt3x_2>|x_1|\}$, they are both calibrated by
\begin{equation*}
\Phi(x)=\tfrac12\left(\begin{smallmatrix}1&\cyan{\mp}\sqrt3\\1&\cyan{\pm}\sqrt3\end{smallmatrix}\right)\text{ if }x\in A_{\pm},\quad
\Phi(x)=\left(\begin{smallmatrix}1&0\\0&0\end{smallmatrix}\right)\text{ if }x\in B_+,\quad
\Phi(x)=\left(\begin{smallmatrix}0&0\\1&0\end{smallmatrix}\right)\text{ if }x\in B_-.
\end{equation*}
\end{examp}
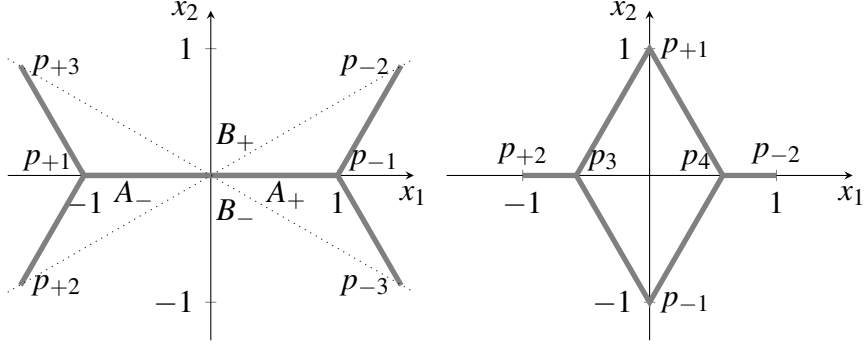
\begin{figure}
\centering
\begin{tikzpicture}
\begin{axis}[
x = 4em, y = 4em,
ymax = 1.3,
ymin = -1.3,
xmax = 1.6,
xmin = -1.6,
axis x line=center,
axis y line=center,
xlabel = {$x_1$},
ylabel = {$x_2$},
x label style={at={(axis description cs:1,.5)},anchor=north},
y label style={at={(axis description cs:.5,1)},anchor=east},
xtick = {-1,1},
xticklabels = {$-1$,$1$},
ytick = {-1,1},
yticklabels = {$-1$,$1$},
]
\addplot[gray,line width=2pt] coordinates {(-1.5,0.866) (-1,0) (1,0) (1.5,0.866)};
\addplot[gray,line width=2pt] coordinates {(-1.5,-0.866) (-1,0) (1,0) (1.5,-0.866)};
\addplot[dotted] coordinates {(-1.6,-0.92) (1.6,0.92)};
\addplot[dotted] coordinates {(-1.6,0.92) (1.6,-0.92)};
\node[anchor=west] at (axis cs:-1.5,0.866) {$p_{\cyan{+3}}$};
\node[anchor=west] at (axis cs:-1.5,-0.866) {$p_{\cyan{+2}}$};
\node[anchor=east] at (axis cs:1.5,0.866) {$p_{\cyan{-2}}$};
\node[anchor=east] at (axis cs:1.5,-0.866) {$p_{\cyan{-3}}$};
\node[anchor=south east] at (axis cs:-1,-0.05) {$p_{\cyan{+1}}$};
\node[anchor=south west] at (axis cs:1,-0.05) {$p_{\cyan{-1}}$};
\node at (axis cs:.6,-.15) {$A_+$};
\node at (axis cs:-.6,-.15) {$A_-$};
\node at (axis cs:.2,.3) {$B_+$};
\node at (axis cs:.2,-.3) {$B_-$};
\end{axis}
\end{tikzpicture}
\begin{tikzpicture}
\begin{axis}[
x = 4em, y = 4em,
ymax = 1.3,
ymin = -1.3,
xmax = 1.6,
xmin = -1.6,
axis x line=center,
axis y line=center,
xlabel = {$x_1$},
ylabel = {$x_2$},
x label style={at={(axis description cs:1,.5)},anchor=north},
y label style={at={(axis description cs:.5,1)},anchor=east},
xtick = {-1,1},
xticklabels = {$-1$,$1$},
ytick = {-1,1},
yticklabels = {$-1$,$1$},
]
\addplot[gray,line width=2pt] coordinates {(-1,0) (-.577,0) (0,1) (.577,0) (1,0)};
\addplot[gray,line width=2pt] coordinates {(-.577,0) (0,-1) (.577,0)};
\node[anchor=west] at (axis cs:0,1) {$p_{\cyan{+1}}$};
\node[anchor=west] at (axis cs:0,-1) {$p_{\cyan{-1}}$};
\node[anchor=south] at (axis cs:-1,0) {$p_{\cyan{+2}}$};
\node[anchor=south west] at (axis cs:-.577,-0.05) {$p_{\cyan{3}}$};
\node[anchor=south east] at (axis cs:.577,-0.05) {$p_{\cyan{4}}$};
\node[anchor=south] at (axis cs:1,0) {$p_{\cyan{-2}}$};
\end{axis}
\end{tikzpicture}%
\caption{Optimal multi-material transport network and notation from \cref{exm:twoSourcesSinks,exm:cycle} (left and right).}
\label{fig:twoSourcesSinks}
\end{figure}

While the optimal transportation networks in (classical) branched transport (for which multi-material transport \cyan{may} be used as a convexification) are known to always have a tree topology \cite{BCM08}
and also in single-material transport (which is equivalent to Wasserstein-$1$ transport) there always exists an optimal transportation network \red{without a loop},
this is no longer true in multi-material transport, as the next example illustrates.

\begin{examp}[Transportation network with cycle]\label{exm:cycle}
As in the previous example, let $m=n=2$ and use the same $h$, $\vec{e}_i$, and $\theta^i$, $i=1,2,3$. Further define
\begin{gather*}\textstyle
\theta^4={-1\choose1},\
p_{\cyan{+1}}={0\choose1},\
p_{\cyan{-1}}={0\choose-1},\
p_{\cyan{+2}}={-1\choose0},\
p_3={-1/\sqrt3\choose0},\
p_4={1/\sqrt3\choose0},\
p_{\cyan{-2}}={1\choose0},\\
F_0=\theta^4(\delta_{p_{\cyan{-1}}}-\delta_{p_{\cyan{+1}}})+\theta^{\cyan{1}}(\delta_{p_{\cyan{-2}}}-\delta_{p_{\cyan{+2}}}).
\end{gather*}
Then an optimizer of \labelcref{DM} is given by (cf.\:\cref{fig:twoSourcesSinks} right)
\begin{equation*}
F=\theta^{\cyan{1}}\otimes\vec{e}_1\mathcal{H}^1\mres[p_{\cyan{+2}},p_{\cyan{3}}]\cup[p_{\cyan{4}},p_{\cyan{-2}}]+\theta^2\otimes\vec{e}_2\mathcal{H}^1\mres[p_3,p_{\cyan{+1}}]\cup[p_{\cyan{-1}},p_{\cyan{4}}]+\theta^3\otimes\vec{e}_3\mathcal{H}^1\mres[p_{\cyan{3}},p_{\cyan{-1}}]\cup[p_{\cyan{+1}},p_{\cyan{4}}],
\end{equation*}
\cyan{because} it is calibrated by
\begin{equation*}
\Phi(x)=\tfrac12\left(\begin{smallmatrix}1&\sqrt3\\1&-\sqrt3\end{smallmatrix}\right).
\end{equation*}
Since $\vec{e}_1,\vec{e}_2,\vec{e}_3$ are the only allowed transport directions by \cref{rem:dualPotential},
one can check that there exists no optimal transportation network between the sources and sinks with a tree topology.
\end{examp}

The maximum possible degree of a junction $x$ in an optimal multi-material transport network depends on the multi-material transport cost $h$ and the space dimension $n$:
By \eqref{eqn:extremalityCalibration} the \cyan{derivative $\Phi(x)\vec{e}$ of $\phi$ at $x$ in direction $\vec{e}$ (assuming that calibration $\Phi=\J\phi$ is totally differentiable in $x$) must lie in the boundary of $\partial h(0)$ for all incident edge directions $\vec{e}$.}
As one can see in \cref{fig:homogenization} left, the cost from that figure allows at most junctions with six incident edges.
There is, however, no fundamental obstruction to arbitrary degrees, as the following example illustrates.

\begin{examp}[Arbitrary degree junction]
Fix $m=n=2$ and let $\vec{e}_1,\ldots,\vec{e}_k\in\mathcal{S}^{1}\cap[0,1]^2$ be pairwise distinct with $k\geq3$.
Let $a_1,\ldots,a_k\in\R\setminus\{0\}$ such that $\sum_{i=1}^ka_i\vec{e}_i=0$,
where we assume without loss of generality that $a_1,\ldots,a_j>0$ and $a_{j+1},\ldots,a_k<0$.
Now pick material vectors $\theta_i=|a_i|\vec{e}_i$ for $i=1,\ldots,k$ as well as lengths $l_1,\ldots,l_k>0$
and define the sink locations $x_i=l_i\vec{e}_i$, $i=1,\ldots,j$, and source locations $x_i=-l_i\vec{e}_i$, $i=j+1,\ldots,k$.
The multi-material transport network
\begin{equation*}
F=\sum_{i=1}^j\theta_i\otimes\vec{e}_i\mathcal{H}^1\mres[0,x_i]\cyan{+}\sum_{i=j+1}^k\theta_i\otimes\vec{e}_i\mathcal{H}^1\mres[\cyan{x_i,0}]
\qquad\text{with boundary $F_0=\partial F=\sum_{i=1}^j\theta_i\delta_{x_i}-\sum_{i=j+1}^k\theta_i\delta_{x_i}$}
\end{equation*}
has a junction in the origin with degree $k$.
Moreover, $F$ is optimal in \labelcref{DM} for the multi-material transport cost $h=|\cdot|$.
Indeed, a corresponding calibration is given by $\Phi=I\in\R^{2\times2}$, the identity matrix,
which is the Jacobian of the linear potential $\phi:x\mapsto x$.
Clearly, we have $\Phi\in\partial H(0)$ since $\Phi B_1(0)=B_1(0)=\partial h(0)$. Moreover, we obtain
\begin{equation*}
\int\phi\cdot\mathrm{d}F_0
=\sum_{i=1}^{j}\theta_i\cdot\phi (x_i)-\sum_{i=j+1}^k\theta_i\cdot\phi(x_i)
=\sum_{i=1}^{k}|a_i|l_i
=\sum_{i=1}^{k}h(\theta_i)\mathcal{H}^1([0,x_i])
=\mathbb{M}_h(F).
\end{equation*}
In fact, the same argument works for any norm $h$ such that the boundary of $\partial h(0)$ is tangential to the unit ball in $\vec{e}_1,\ldots,\vec{e}_k$,
e.g.\:the largest norm $h$ with $h(\vec{e}_1)=\ldots=h(\vec{e}_k)=1$.
Note that, by choosing $a_1,\ldots,a_k$ appropriately, any number of incoming or outgoing edges between $1$ and $k-1$ is possible.
\end{examp}
The generically occurring degree of a junction, given a multi-material transport cost $h$ and space dimension $n$, however, will be much lower.
For instance, if $h=|\cdot|_1$ is the $\ell^1$-norm, then all different materials are transported independently from each other
(there is no discount for transporting different materials together)
so that each material is transported along straight lines from its sources to its sinks.
Since straight lines generically do not cross in more than two spatial dimensions,
there are generically no interior junctions in an optimal multi-material transport network.
This changes if $\partial h(0)\subset\R^m$ contains some hyperellipsoid \red{$\Phi B_1(0)$} touching the boundary \red{of $\partial h(0)$} in $k>m$ points with pairwise nonparallel position vectors.
We expect that in that case in dimensions $n\geq m$ a junction in an optimal multi-material transport network generically has degree $\lfloor m(m+1)/(2m-2)\rfloor$:
The touching points define an $m$-dimensional $k$-stencil (modelling an allowed junction of degree $k$) up to rigid motion in $\R^n$.
\red{Moreover, all directions of that stencil lie in an $m$-dimensional subspace of $\R^n$ (since the preimage of the boundary of $\partial h(0)$ under $\Phi$ intersects $B_1(0)$ only in an $m$-dimensional subspace). Therefore}
we may without loss of generality restrict to $n=m$.
Then given $\ell$ source- and sink-positions, we would like to rigidly position the stencil (or equivalently the source and sink configuration)
such that each source and sink lies on a stencil branch\red{, i.e.\:we seek a corresponding rotation $R\in\mathrm{SO}(m)$ and translation $t\in\R^m$ of the stencil}.
\red{That a point lies on a line can be expressed by exactly $m-1$ equations, hence we get}
$m-1$ constraints for each source and sink, while the rigid motions \red{$\mathrm{SO}(m)\times\R^m$} form a $m(m+1)/2$-dimensional manifold of \red{sought} degrees of freedom,
resulting in the estimate $\ell\leq m(m+1)/(2m-2)$.
\section{Equality of the masses $\mathcal{M}_h,|.|_H$, and $\mathbb{M}_h$}
\label{equalityofmasses}
\red{We first prove $\mathcal{M}_h=|.|_H$ in \cref{sec:equalityCurrentsMeasures} and then $|.|_H=\mathbb{M}_h$ in \cref{viaminimizers}.}

\subsection{$\mathcal{M}_h=|.|_H$ via duality in the definitions}\label{sec:equalityCurrentsMeasures}
\begin{proof}[Proof of \crefthmpart{equalitymasses}{item1}]
If $\omega\in\Omega_{m,c}^1(\R^n)$ and $x\in\R^n$, then we identify $\omega(x)\in\Lambda_m^1\R^n$ with its natural representation matrix \red{in $\R^{m\times n}$}. Using \cref{Hstarop} we obtain $|\omega(x)|_h=H_*(\omega(x))$, thus $|\omega|_h^\infty=\sup_{\R^n}H_*(\omega)$. Now let $F\in \mathcal{M}_1^m(\R^n)$ be arbitrary. For all $\varepsilon>0$ there exists a compact and convex set $K_\varepsilon\subset\R^n$ such that $|F|(\R^n\backslash K_\varepsilon)<\varepsilon$ by the regularity of $|F|$. We can choose $K_\varepsilon$ to equal the closure of $B_{R_\varepsilon}(0)$, where $R_\varepsilon\to\infty$. In particular, we can assume $|(F\mres (\R^n\backslash K_\varepsilon))(\omega )|\leq\varepsilon$ for all $\omega\in\Omega_{m,c}^1(\R^n)$ with $|\omega|_h^\infty\leq 1$. Thus, by density and \cref{glem} we obtain
\begin{multline*}
\mathcal{M}_h(F)=\sup_{|\omega|_h^\infty\leq 1}F(\omega )=\sup_{|\omega|_h^\infty\leq 1}(F\mres K_\varepsilon)(\omega )+(F\mres (\R^n\backslash K_\varepsilon))(\omega )\geq \sup_{\substack{\Phi\in C(K_\varepsilon;\R^{m\times n})\\ H_*(\Phi)\leq 1\textup{ in }K_\varepsilon}}\langle\Phi,F\mres K_\varepsilon\rangle-\varepsilon\\
=\red{\iota}_{\partial H(0),K_\varepsilon}^*(F\mres K_\varepsilon)-\varepsilon=|F\mres K_\varepsilon|_H-\varepsilon.
\end{multline*} 
Hence, we have $\mathcal{M}_h(F)\geq\liminf_{\varepsilon\to 0}|F\mres K_\varepsilon|_H\geq |F|_H$ using the weak-$*$ lower semicontinuity of $|\cdot|_H$. For the reverse inequality note that $F=\sum_QF\mres Q$ by the $\sigma$-additivity of $F$, where the sum is over half-open unit cubes $Q=[z_1,z_1+1)\times\ldots\times [z_n,z_n+1)$ with $z_i\in\mathbb{Z}$. Again, by density and \cref{glem},
\begin{equation*}
\mathcal{M}_h(F)=\sup_{|\omega|_h^\infty\leq 1}F(\omega )\leq\sum_Q\sup_{\substack{\Phi\in C(\overline{Q};\R^{m\times n})\\ H_*(\Phi)\leq 1\textup{ in }\overline{Q}}}\langle\Phi,F\mres Q\rangle=\sum_Q\red{\iota}_{\partial H(0),\overline{Q}}^*(F\mres Q)=\sum_Q|F\mres Q|_H=|F|_H,
\end{equation*}
which finishes the proof.
\end{proof}
\subsection{$|.|_H=\mathbb{M}_h$ via minimizers of problems \labelcref{DM,F}}
\label{viaminimizers}
\begin{proof}[Proof of \crefthmpart{equalitymasses}{item2}]
Let $F\in\mathcal{F\!\!M}_1^m(\R^n)$.

\underline{$|F|_H\leq\mathbb{M}_h(F)$:} By definition of $\mathbb{M}_h$ there exists a sequence $(P_j)\subset\mathbb{P}_1^m(\R^n)$ with $\mathbb{F}(P_j-F)\to 0$ and $\mathbb{M}_h(P_j)\to\mathbb{M}_h(F)$. Since $\mathbb{M}_h(P_j)=|P_j|_H$ is bounded, we have $P_j\ws F$ up to a subsequence using the Banach\textendash Alaoglu theorem. Therefore, the lower semicontinuity of $|.|_H$ (\cref{glem}) implies
\begin{equation*}
|F|_H\leq\liminf_j|P_j|_H=\liminf_j\mathbb{M}_h(P_j)=\mathbb{M}_h(F).
\end{equation*}

\underline{$\mathbb{M}_h(F)\leq |F|_H$ for the case $F\in\mathcal{N}_1^m(\R^n)$:} STEP 1 (Preparation): For $\delta>0$ define the lattice
	\begin{equation*}
	L_\delta=\{ x\in\R^n\:|\:x_i\in\delta\mathbb{Z}\textup{ for some }i \}.
	\end{equation*} 
	By \cite[p.\:395]{F} we can assume that the positioning of $L_\delta$ is such that (otherwise translate $L_\delta$ or $F$)
		\begin{equation}
		\label{mass}
		|F_i|(L_\delta)+|\partial F_i|(L_\delta)=0
		\end{equation}
		for all $i=1,\ldots, m$ and all but countably many $\delta>0$. We restrict to values of $\delta$ for which this is satisfied.
	Moreover, define the set of half-open $\delta$-hypercubes whose boundaries are contained in $L_\delta$,
	\begin{equation*}
	C_\delta=\left\{ \red{\bigtimes_{i=1}^n}[z_i,z_i+\delta)\:\middle|\:z_i\in \delta\mathbb{Z}  \right\}.
	\end{equation*}
%
	For all $Q\in C_\delta$ we define
		\begin{equation*}
    \red{\vec{\mu}_Q=\partial(F\mres Q).}
		\end{equation*}
We have $\vec{\mu}_Q\in\mathcal{N}_0^m(\R^n)$ for a.e.\:$\delta>0$ (and for the rest of the proof we restrict to such values of $\delta$):
indeed, by \cite[pp.\:395\:\&\:440]{F} we get
			\begin{equation*}
			(\vec{\mu}_Q)_i=(\partial F_i)\mres Q+\sum_{j=1}^{2n}[F_i,S_j^Q,\nu_j^Q],\qquad i=1,\ldots,m,
			\end{equation*}
where $[F_i,S_j^Q,\nu_j^Q]\in\mathcal{F}_0^1(\R^n)$ is the so-called slicing of $F_i$ through the $j$-th $(n-1)$-dimensional boundary hyperface $S_j^Q$ of $Q$ with respect to unit outward normal $\nu_j^Q\perp\textup{aff}(S_j^Q)$
(which by \cite[pp.\:395\:\&\:440]{F} can be expressed as $[F_i,S_j^Q,\nu_j^Q] = ( (\partial F_i)\mres H_j^Q-\partial(  F_i\mres H_j^Q) )\mres S_j^Q$ for $H_j^Q=\textup{aff}( S_j^Q )+(0,\infty)\nu_j^Q$),
and it holds $[F_i,S_j^Q,\nu_j^Q] \in\mathcal{N}_0^1(\R^n)$ by \cite[Thm.\:4.3.2 (3)]{F} for a.e.\:$\delta>0$.
			
		STEP 2 (Definition of ``locally optimal'' $F_\delta$): Using \cref{projlem,exmin,eqforopt} there exist minimizers $F_Q\in\mathcal{N}_{1,\overline{Q}}^m(\R^n)$ of
		\begin{equation*}
		\min_{\partial G=\vec{\mu}_Q}\mathbb{M}_h(G)
		\end{equation*}
		that satisfy $\mathbb{M}_h(F_Q)=|F_Q|_H$. Now define
		\begin{equation*}
		F_\delta=\sum_{Q\in C_\delta}F_Q.
		\end{equation*}

		STEP 3 ($\mathbb{N}_h(F_\delta)$ uniformly bounded): Using the optimality of $F_Q$ \red{and \cref{eqforopt}} we have
		\begin{equation}\label{eqn:constructionHasSmallerMass}
		\mathbb{M}_h(F_\delta)\leq\sum_{Q\in C_\delta}\mathbb{M}_h(F_Q)=\sum_{Q\in C_\delta}|F_Q|_H\leq\sum_{Q\in C_\delta }|F\mres Q|_H=|F|_H.
		\end{equation}
		Hence, we get that $\mathbb{N}_h(F_\delta)$ is uniformly bounded in $\delta$ (we clearly have $\partial F_\delta=\partial F$).

		STEP 4 ($F_\delta \ws F$): Without loss of generality, we can assume $m=1$. Fix some $\phi\in C_c^\infty(\R^n;\R^n)$ (by the Stone\textendash Weierstrass theorem it is sufficient to take regular $\phi$ because $|F_\delta-F|(\R^n )$ is uniformly bounded in $\delta$). Write $G_\delta=F_\delta-F$.
		Using that $\phi$ is Lipschitz, there exists a constant $C=C(\phi)>0$ such that $\phi(x)=\phi(x_Q)+\delta Ct(x)$ for all $x\in Q$ and $Q\in C_\delta$, where $x_Q$ denotes the centre of $Q$ and \red{$t:\R^n\to[-1,1]^n$ is continuous on each $Q$}. Thus, we get
		\begin{equation*}
		\left|\int\phi\cdot\mathrm{d}G_\delta\right|\leq\sum_{Q\in C_\delta}\left|\int_Q\phi(x_Q)\cdot\mathrm{d}G_\delta\right|+\delta C\left|  \int_Q t\cdot\mathrm{d}G_\delta\right|
\red{\leq\sum_{Q\in C_\delta}\left|\int_Q\phi(x_Q)\cdot\mathrm{d}G_\delta\right|+\delta C\sqrt n\mathbb{M}(G_\delta)}.
		\end{equation*}
		The \red{last summand} clearly converges to zero since the mass of $G_\delta$ is uniformly bounded. Now \red{let $\psi_Q(x)=\chi_Q(x)\phi(x_Q)\cdot x$ with $\chi_Q$ a smooth cutoff function taking the value $1$ on $Q$ and} estimate
		\begin{equation*}
		\red{\left|\int_{Q}\phi(x_Q)\cdot\mathrm{d}G_\delta\right|
    =\left|\int_{Q}\nabla\psi_Q\cdot\mathrm{d}G_\delta\right|
    =\left|\langle\nabla\psi_Q,F_Q-F\mres Q\rangle\right|
    =\left|\langle\psi_Q,\partial(F_Q-F\mres Q)\rangle\right|
    =0,}
		\end{equation*}
    \red{since $\partial F_Q=\vec{\mu}_Q=\partial(F\mres Q)$. Hence the first summand vanishes, and we have proven $F_\delta \ws F$.}

		FINAL STEP: By $F_\delta -F\ws 0$ and and the uniform boundedness of $\mathbb{N}_h(F_\delta)$ we get $\mathbb{F}(F_\delta -F)\to 0$ up to a subsequence (\cref{Compactness}). 
    The lower semicontinuity of $\mathbb{M}_h$ \red{and \eqref{eqn:constructionHasSmallerMass} then imply}
		\begin{equation*}
    \red{\mathbb{M}_h(F)
    \leq\liminf_\delta\mathbb{M}_h(F_\delta)
    \leq|F|_H}.
		\end{equation*}

\underline{$\mathbb{M}_h(F)\leq |F|_H$ for the case $F\in\mathcal{F\!\!M}_1^m(\R^n)\backslash\mathcal{N}_1^m(\R^n)$:} By the deformation theorem \cite{W992} there exists $(P_j)\subset\mathbb{P}_1^m(\R^n)$ with $\mathbb{F}(F-P_j)\to 0$ and $\mathbb{F}(\partial F-\partial P_j)\to 0$. By \cref{flatform} there are $A_j\in\mathcal{F\!\!M}_{\red{0}}^m(\R^n),B_j\in\mathcal{F\!\!M}_{\red{1}}^m(\R^n)$ with $\partial F-\partial P_j=A_j+\partial B_j$ and $\mathbb{M}_h(A_j)+\mathbb{M}_h(B_j)\to 0$. Hence, we obtain $\mathbb{M}_h(\partial (F-B_j))=\mathbb{M}_h(A_j+\partial P_j)\leq \mathbb{M}_h(A_j)+\mathbb{M}_h(\partial P_j)<\infty$. Thus, we have $F-B_j\in\mathcal{N}_1^m(\R^n)$, therefore $\mathbb{M}_h(F-B_j)\leq |F-B_j|_H$ by the second case. Finally,
\begin{equation*}
\mathbb{M}_h(F)-\mathbb{M}_h(B_j)\leq\mathbb{M}_h(F-B_j)=|F-B_j|_H\leq|F|_H+|B_j|_H.
\end{equation*}
Letting $j\to 0$ we obtain the desired inequality using $|B_j|_H\leq\mathbb{M}_h(B_j)$ (cf.\:beginning of proof).
\end{proof}

\section{Acknowledgements}
This work was supported by the Deutsche Forschungsgemeinschaft (DFG, German Research Foundation) under the priority program SPP 1962, Grant WI 4654/1-1, and under Germany’s Excellence Strategy EXC 2044-390685587, Mathematics Münster: Dynamics-Geometry-Structure. B.W.’s and J.L.’s research was supported by the Alfried Krupp Prize for Young University Teachers awarded by the Alfried Krupp von Bohlen und Halbach-Stiftung. B.S.’s research was supported by the Emmy Noether Programme of the DFG, Grant SCHM 3462/1-1.

\printbibliography
\end{document}